\documentclass[11pt,reqno,tbtags]{amsart}
\usepackage{amssymb}
\usepackage{natbib}
\bibpunct[, ]{[}{]}{;}{n}{,}{,}

\ifx\pdfoutput\undefined
\usepackage[dvips]{graphicx}
\else
\usepackage[pdftex]{graphicx}
\usepackage{epsfig}
\usepackage{latexsym}

\numberwithin{equation}{section}

\allowdisplaybreaks

\newtheorem{theorem}{Theorem}[section]
\newtheorem{lemma}[theorem]{Lemma}

\theoremstyle{definition}

\newtheorem{remark}[theorem]{Remark}

\theoremstyle{remark}

\newenvironment{romenumerate}{\begin{enumerate}
 }{\end{enumerate}}

\newcounter{oldenumi}
{\setcounter{oldenumi}{\value{enumi}}
\begin{romenumerate} \setcounter{enumi}{\value{oldenumi}}}
{\end{romenumerate}}

\newcounter{thmenumerate}
\newenvironment{thmenumerate}
{\setcounter{thmenumerate}{0}%
 \def\item{\par
 \refstepcounter{thmenumerate}\textup{(\roman{thmenumerate})\enspace}}
}
{}

\newcounter{xenumerate}   

\newcommand{\refT}[1]{Theorem~\ref{#1}}

\newcommand{\refL}[1]{Lemma~\ref{#1}}
\newcommand{\refR}[1]{Remark~\ref{#1}}
\newcommand{\refS}[1]{Section~\ref{#1}}

\newcommand{\refF}[1]{Figure~\ref{#1}}
\newcommand{\refApp}[1]{Appendix~\ref{#1}}

\newcommand{\refand}[2]{\ref{#1} and~\ref{#2}}

\begingroup
  \divide\count255 by 60
  \multiply\count255 by -60
  \advance\count255 by \time
  \ifnum \count255 < 10 \xdef\klockan{\the\count1.0\the\count255}
  \else\xdef\klockan{\the\count1.\the\count255}\fi
\endgroup

\newcommand{\sumj}{\sum_{j=0}^\infty}
\newcommand{\sumk}{\sum_{k=0}^\infty}
\newcommand{\sumki}{\sum_{k=1}^\infty}
\newcommand{\sumkcc}{\sum_{k=200}^\infty}
\newcommand{\sumkzz}{\sum_{k=200}^\infty}
\newcommand{\sumjc}{\sum_{j=100}^\infty}
\newcommand{\sumell}{\sum_{\ell=0}^\infty}

\newcommand\set[1]{\ensuremath{\{#1\}}}

\newcommand\xpar[1]{(#1)}
\newcommand\bigpar[1]{\bigl(#1\bigr)}
\newcommand\Bigpar[1]{\Bigl(#1\Bigr)}
\newcommand\biggpar[1]{\biggl(#1\biggr)}
\newcommand\lrpar[1]{\left(#1\right)}
\newcommand\bigsqpar[1]{\bigl[#1\bigr]}
\newcommand\Bigsqpar[1]{\Bigl[#1\Bigr]}

\newcommand\lrsqpar[1]{\left[#1\right]}

\newcommand\bigabs[1]{\bigl|#1\bigr|}
\newcommand\abs[1]{|#1|}
\newcommand\Bigabs[1]{\Bigl|#1\Bigr|}

\newcommand\lrabs[1]{\left|#1\right|}
\def\rompar(#1){\textup(#1\textup)}    
\newcommand\xfrac[2]{#1/#2}

\newcommand\parfrac[2]{\lrpar{\frac{#1}{#2}}}

\newcommand\Bigparfrac[2]{\Bigpar{\frac{#1}{#2}}}

\def\xexp(#1){e^{#1}}

\newcommand\ntoo{\ensuremath{{n\to\infty}}}
\newcommand\Ntoo{\ensuremath{{N\to\infty}}}

\newcommand\norm[1]{\|#1\|}

\newcommand\downto{\searrow}
\newcommand\upto{\nearrow}

\newcommand\ie{i.e.\spacefactor=1000}
\newcommand\eg{e.g.\spacefactor=1000}

\newcommand\etc{etc.\spacefactor=1000}
\newcommand\cf{cf.\spacefactor=1000}
\newcommand{\as}{a.s.\spacefactor=1000}

\newcommand\ii{{i}}

\newcommand{\tend}{\longrightarrow}
\newcommand\dto{\overset{\mathrm{d}}{\tend}}

\newcommand\eqd{\overset{\mathrm{d}}{=}}

\newcommand\bbR{\mathbb R}
\newcommand\bbC{\mathbb C}
\newcommand\bbN{\mathbb N}  

\newcounter{CC}
\newcounter{cc}

\renewcommand\Re{\operatorname{Re}}
\renewcommand\Im{\operatorname{Im}}

\newcommand\E{\operatorname{\mathbb E{}}}
\renewcommand\P{\operatorname{\mathbb P{}}}
\newcommand\Var{\operatorname{Var}}

\newcommand\ga{\alpha}
\newcommand\gb{\beta}
\newcommand\gd{\delta}

\newcommand\gf{\varphi}
\newcommand\gam{\gamma}
\newcommand\gG{\Gamma}

\newcommand\gl{\lambda}

\newcommand\go{\omega}

\newcommand\gs{\sigma}

\newcommand\gth{\theta}

\newcommand\ett[1]{\boldsymbol1[#1]} 

\newcommand\qq{^{1/2}}
\newcommand\qqq{^{1/3}}
\newcommand\qqqq{^{1/4}}
\newcommand\qqw{^{-1/2}}
\newcommand\qqqw{^{-1/3}}
\newcommand\qqqqw{^{-1/4}}
\newcommand\qw{^{-1}}
\newcommand\qww{^{-2}}
\newcommand\qqqb{^{2/3}}
\newcommand\qqqbw{^{-2/3}}
\newcommand\qqc{^{3/2}}
\newcommand\qqcw{^{-3/2}}

\renewcommand{\=}{:=}

\newcommand\intot{\int_0^t}
\newcommand\intoo{\int_0^\infty}
\newcommand\intoooo{\int_{-\infty}^\infty}
\newcommand\intx[1]{\int_{#1-\infty\ii}^{#1+\infty\ii}}
\newcommand\intgs{\intx{\gs}}
\newcommand\oi{[0,1]}

\newcommand\dd{\,\textup{d}}

\newcommand\rhs{right-hand side}

\newcommand\Ai{\mathrm{Ai}}
\newcommand\AI{\mathrm{AI}} 
\newcommand\Bi{\mathrm{Bi}}
\newcommand\Gi{\mathrm{Gi}}
\newcommand\Hi{\mathrm{Hi}}

\newcommand\AS[1]{\cite[#1]{AS}}
\newcommand\glx{a}
\newcommand\intakoo{\int_{a_k}^\infty}

\newcommand\eith{e^{\ii \gth}}
\newcommand\eigf{e^{\ii \gf}}
\newcommand\intg{\int_{\Gamma}}
\newcommand\intgg{\int_{\gGG}}
\newcommand\intgn{\int_{\Gamma_N}}

\newcommand\zz{|z|}

\newcommand\vx[1]{V_{#1}}
\newcommand\vs{\vx{s}}
\newcommand\ws{W_{s}}
\newcommand\tM{\widetilde M}
\newcommand\xc{_{\gamc}}
\newcommand\ftm{f_{\tau,\tM}}
\newcommand\fm{f_{\tM}}
\newcommand\gc{g\xc}
\newcommand\cgc{\check g\xc}
\newcommand\kc{k\xc}
\newcommand\hcx[1]{h_{\gamc,\,#1}}
\newcommand\hcyc{\hcx{y+\gamc s^2}}
\newcommand\hca{\hcx{a}}
\newcommand\fcqqq{(4\gamc)\qqq}
\newcommand\fcqqqw{(4\gamc)\qqqw}
\newcommand\bccqqq{(2\gamc^2)\qqq}
\newcommand\bccqqqw{(2\gamc^2)\qqqw}
\newcommand\FT{\widehat}
\newcommand\gfc{c}
\newcommand\EN{S^{N}}
\newcommand\EM{S^{M}}
\newcommand\gGG{\gG'}
\newcommand\lcdots{\ldots\cdot}
\newcommand\gamc{\gamma}
\newcommand\Ex{\E_x}
\newcommand\tg{\widetilde g}
\newcommand\thx{\widetilde h}
\newcommand\tk{\widetilde k}
\newcommand\hthx{\widehat h}
\newcommand\htk{\widehat k}
\newcommand\ftt{f_\tau}
\newcommand\fttt{f_\tau(t)}
\newcommand\citetq[2]{\citeauthor{#2} \cite[{\frenchspacing #1}]{#2}} 
\newcommand\gfo{\gf_0}
\newcommand\gfoo{\gf_\infty}

\newcommand{\dt}{2^{1/3}}
\newcommand{\dtc}{2^{2/3}}

\newcommand{\Fi}{\varphi}
\newcommand{\II}{\infty}
\newcommand{\lb}{\left [}
\newcommand{\rb}{\right ]}

\newcommand{\kl}{[\hspace{-0.5 mm}[}
\newcommand{\kr}{]\hspace{-0.5 mm}]}
\newcommand{\non}{\nonumber}
\newcommand{\ba}{\begin{array}}
\newcommand{\ea}{\end{array}}
\newcommand{\bit}{\begin{itemize}}
\newcommand{\eit}{\end{itemize}}
\newcommand{\ben}{\begin{enumerate}}
\newcommand{\een}{\end{enumerate}}
\newcommand{\bal}{\begin{align}}
\newcommand{\bals}{\begin{align*}}
\newcommand{\bs}{\begin{skip}}
\newcommand{\eal}{\end{align}}
\newcommand{\eals}{\end{align*}}
\newcommand{\es}{\end{skip}}

\newcommand{\maple}{\texttt{Maple}}

\newcommand\urladdrx[1]{{\urladdr{\def~{{\tiny$\sim$}}#1}}}

\begin{document}
\title
{The maximum of Brownian motion with parabolic drift}

\thanks{A major part of this research was
done while the authors all visited Institut Mittag-Leffler, Djursholm,
Sweden, during the program 'Discrete Probability' 2009.}

\date{January 28, 2009} 

\author{Svante Janson}
\address{Department of Mathematics, Uppsala University, PO Box 480,
SE-751~06 Uppsala, Sweden}
\email{svante.janson@math.uu.se}
\urladdrx{http://www.math.uu.se/~svante/}

\author{Guy Louchard}
\address{Universit\'e Libre de Bruxelles,
D\'epartement d'Informatique, CP 212, Boulevard du Triomphe, B-1050
Bruxelles, Belgium}
\email{louchard@ulb.ac.be}

\author{Anders Martin-L\"of}
\address{Department of Mathematics, Stockholm University,
SE-10691 Stockholm, Sweden}
\email{andersml@math.su.se}


\begin{abstract} 
We study the maximum of a Brownian motion with a parabolic drift; this
is a random variable that often occurs as a limit of the maximum of
discrete processes whose expectations have a maximum at an interior
point.
We give series expansions and integral formulas for the distribution
and the first two moments, together with numerical values to high precision.
\end{abstract}

\maketitle

\section{Introduction}\label{S:intro}

Let $W(t)$ be a two-sided Brownian motion with $W(0)=0$;
i.e., $(W(t))_{t\ge0}$ and
$(W(-t))_{t\ge0}$ are two independent standard Brownian
motions.
We are interested in the process 
\begin{equation}\label{wg}
  W_\gamma(t)\=W(t)-\gamma t^2
\end{equation}
for a given $\gamma>0$,
and in particular in its maximum
\begin{equation}\label{m}
  M_\gamma\=\max_{-\infty<t<\infty} W_\gamma(t)
=\max_{-\infty<t<\infty} \bigpar{W(t)-\gamma t^2}.
\end{equation}
We also consider the corresponding one-sided maximum
\begin{equation}\label{n}
  N_\gamma\=\max_{0\le t<\infty} W_\gamma(t)
=\max_{0\le t<\infty} \bigpar{W(t)-\gamma t^2}.
\end{equation}
Since the restrictions of $W$ to the positive and negative
half-axes are independent, we have the relation
\begin{equation}\label{mnn}
  M_\gamma\eqd\max(N_\gamma,N_\gamma')
\end{equation}
where $N_\gamma'$ is an independent copy of $N_\gamma$.

Note that (\as) $W_\gam\to-\infty$ as
$t\to\pm\infty$, so the maxima in \eqref{m}
and \eqref{n} exist and are finite; moreover, they are
attained at unique points and $M_\gam,N_\gam>0$.
It is easily seen (e.g., by Cameron--Martin) that $M_\gam$
and $N_\gam$ have absolutely continuous distributions.

\subsection{Scaling}

For any $a>0$, $W(at)\eqd a\qq W(t)$ (as processes on
$(-\infty,\infty)$), and thus
  \begin{align}
  M_\gam&= \max_{-\infty<t<\infty} \bigpar{W(at)-\gamma (at)^2}
\notag\\&
\eqd \max_{-\infty<t<\infty} \bigpar{a\qq W(t)-a^2\gamma t^2}
=a\qq M_{a\qqc\gam},	\label{b1m}
\intertext{and similarly}
N_\gam&\eqd a\qq N_{a\qqc\gam}. \label{b1n}
 \end{align}
The parameter $\gam$ is thus just a scale parameter, and it suffices
to consider a single choice of $\gam$. We
choose the normalization $\gam=1/2$ as the standard case, and
write $M\=M_{1/2}$, $N\=N_{1/2}$. In
general, \eqref{b1m}--\eqref{b1n} with $a=(2\gam)\qqqbw$ yield
\begin{equation}\label{scale}
  M_\gam \eqd (2\gam)\qqqw M, \qquad
  N_\gam \eqd (2\gam)\qqqw N.
\end{equation}

\begin{remark}
  More generally, if $a,b>0$, then
  \begin{equation}
	\label{wab}
 \max_{-\infty<t<\infty} \bigpar{aW(t)-b t^2}
=aM_{b/a}
\eqd\Bigparfrac{a^4}{2b}\qqq M.
  \end{equation}
\end{remark}

The purpose of this paper is to provide formulas for the distribution
function of $M$ and, in particular, its moments.
The main results are given in \refS{Smain}, with proofs and further
details in Sections \ref{SpfT1}--\ref{SpfTEM}. The numerical
computations are discussed in \refS{Snum}. We use many more or less
well-known results for Airy functions; for convenience we have
collected them in Appendices \ref{AiryA}--\ref{AiryB}. Finally,
\refApp{AD} discusses an interesting integral equation, while \refApp{Aalt}
contains an alternative proof of an important formula used in  our proofs.

\subsection{Background}

The random variable $M$ 
is studied by Barbour
\cite{B75}, Daniels and Skyrme \cite{DS85} and Groeneboom \cite{Groeneboom}.
It arises as a natural limit distribution
in many different problems, and in many related problems its expectation $\E M$
enters in a second order term for the asymptotics of means
or in improved normal approximations.
For various examples  and general results,
see for example Daniels \cite{D74,D89}, Daniels and Skyrme \cite{DS85},
Barbour \cite{B75,B81}, 
Smith \cite{Smith82},
Louchard,  Kenyon and Schott \cite{LKS97},
Steinsaltz \cite{Steinsaltz},
Janson \cite{SJ198}.
As discussed in several of these papers, the appearance of $M$
in these limit results
can be explained as follows, ignoring technical conditions:
Consider the maximum over time $t$ of a random process $X_n(t)$,
defined on a compact interval $I$, for example $\oi$, 
such that as $\ntoo$, the mean $\E X_n(t)$,
after scaling, converges to 
deterministic function $f(t)$, and that the
fluctuations $X_n(t)-\E X_n(t)$ are of smaller order and,
after a different scaling, converge to a gaussian process $G(t)$.
If we assume that $f$ is continuous on
$I$ and has a unique maximum at a point $t_0\in I$, then the maximum
of the process $X_n(t)$ is attained close to $t_0$. 
Assuming that $t_0$ is an interior point of $I$ and that $f$ is
twice differentiable at $t_0$ with $f''(t_0)\neq0$, we can locally at
$t_0$ approximate $f$ by a parabola and $G(t)-G(t_0)$ by a two-sided
Brownian motion (with some scaling), and thus $\max_t X_n(t)-X_n(t_0)$
is approximated by a scaling constant times the variable $M$,
see Barbour \cite{B75}.
In the typical case where the mean of $X_n(t)$ is of order $n$ and the
Gaussian fluctuations are of order $n\qq$, it is easily seen that
the correct scaling is that
$n\qqqw(\max_t X_n(t)-X_n(t_0))\dto cM$, for some $c>0$,
which for the mean 
gives $\E\max_t X_n(t)=n f(t_0)+n\qqq c\E M+o(n\qqq)$, see
\cite{B75,D74,D89}. 
As examples of applications in algorithmic and data structures analysis,  
this type of asymptotics appears in 
the analysis of linear lists, priority queues and dictionaries 
\cite{L87,LKS97} 
and in a sorting algorithm
\cite{SJ198}.

\section{Main results}\label{Smain}

The mean of $M$ can be expressed as integrals involving the Airy
functions $\Ai$ and $\Bi$, for example as follows.
\begin{theorem}[\citet{DS85}]
  \label{TEM}
  \begin{align}
 \E M&=
-\frac{2\qqqw}{2\pi\ii}\int_{-\ii\infty}^{\ii\infty} \frac{z\dd z}{\Ai(z)^2}
=\frac{2\qqqw}{2\pi\ii}\int_{-\infty}^{\infty} \frac{y\dd y}{\Ai(\ii y)^2}
\label{em1}\\	
&=
{2\qqqb}\int_{0}^{\infty} 
\frac{{\Ai(t)^2+\sqrt3\Ai(t)\Bi(t)}}
{{\Ai(t)^2+\Bi(t)^2}}\,\dd t
\label{em3}\\	
&=
{2\qqqb}\Re\lrpar{\bigpar{1+\ii\sqrt3}\int_{0}^{\infty} 
\frac{{\Ai(t)}}
{\Ai(t)+\ii\Bi(t)}\,\dd t}
\label{em4}\\	
&=
\frac{2\qqqb}{\pi}\int_{0}^{\infty} 
\frac{{\sqrt3\Bi(t)^2-\sqrt3\Ai(t)^2+2\Ai(t)\Bi(t)}}
{\bigpar{\Ai(t)^2+\Bi(t)^2}^2}\,t\dd t.
\label{em2}	
  \end{align}
\end{theorem}
The expressions \eqref{em1} and  
\eqref{em3}
(unfortunately with typos in the latter)
are given by \citet{DS85}.
Since detailed proofs of the formulas are not given there, we for
completeness give a complete proof in \refS{SpfTEM}. (The
proof includes a direct analytical verification of the equivalence of
\eqref{em1} and \eqref{em3}, which was left open in \cite{DS85}.)

By \eqref{Aioo} and \eqref{Bioo},
$|\Ai(\ii y)|$ increases superexponentially as $y\to\pm\infty$, 
while
$\Ai(t)$ decreases superexponentially and
$\Bi(t)$ increases superexponentially as
$t\to\infty$; hence, the integrands in the integrals in \refT{TEM}
all decrease superexponentially and the integrals converge rapidly, so
they are suited for numerical calculations.
We obtain by numerical integration (using \maple{}),  improving the numerical
values in \cite{B75,B81,DS85,D89},
\begin{equation}\label{em0}
 \E M
= 
 0.99619\, 30199\, 28363\, 11660\, 37766\dots 
\end{equation}

We do not know any similar integral formulas for the second 
moment of $M$ (or higher moments). Instead we give expressions
using infinite series, summing over 
the zeros $\glx_k$, $k\ge1$,  of the Airy function, see \refApp{AiryA}.
Recall that these zeros all are real and negative, so we have
$0>\glx_1>\glx_2>\dots$, see
\AS{(10.4.94)} and \refApp{AiryA}; note that
$\abs{a_k}\asymp k\qqqb$, see \eqref{gl}.
(We use $x_n\asymp y_n$, for two sequences of positive
numbers $x_n$ and $y_n$, to denote that $0<\liminf_\ntoo x_n/y_n
\le \limsup_\ntoo x_n/y_n <\infty$; this is also denoted
$x_n=\Theta(y_n)$.)

We first introduce more notation.
Let $F_N(x)$ be the distribution function of $N$, \ie,
$F_N(x)\=\P(N\le x)$, and let
 $F_M(x)$ be the distribution function of $M$; further, let
$G(x)=1-F_N(x)=\P(N>x)$ and $G_M(x)=1-F_M(x)=\P(M>x)$
be the corresponding tail probabilities.
Then, by \eqref{mnn},
\begin{equation}\label{Fm}
  F_M(x)\=\P(M\le x)=\P(N\le x)^2=F_N(x)^2.
\end{equation}
and, equivalently,
\begin{equation}\label{Gm}
  G_M(x)\=1-(1-G(x))^2=2G(x)-G(x)^2.
\end{equation}

If we know $G(x)$, we thus know the distribution of both
$N$ and $M$, and we can compute moments by
\begin{align}
  \E N^p &= p\intoo x^{p-1} G(x)\dd x,
\label{enp}
\\
  \E M^p &= p\intoo x^{p-1} G_M(x)\dd x
        = p\intoo x^{p-1} \bigpar{2G(x)-G(x)^2}\dd x.
\label{emp}
\end{align}

Two formulas for the distribution function are given in the following
theorem. Others are given in \eqref{sal2} and \refL{LJ2}.
The proof is given in \refS{SpfT1}.

\begin{theorem}
  \label{T1}
The distribution functions of $M$ and $N$ are
$F_M(x)=(1-G(x))^2$ and $F_N(x)=1-G(x)$, where
\begin{equation}\label{guy}
G(x)=
 \pi\sumki
\frac{\Hi(\glx_k)}
{\Ai'(\glx_k)} 
{\Ai(\glx_k+2\qqq x)},
\qquad x>0.
\end{equation}
The sum converges conditionally but not absolutely for every $x>0$.
Alternatively, with an absolutely convergent sum,
for $ x\ge0$,
\begin{equation}\label{guy1}
G(x)=
\frac{\Ai(2\qqq x)}{\Ai(0)}
+\sumki 
\frac{\pi\Hi(\glx_k)+\glx_k\qw}
{\Ai'(\glx_k)} 
{\Ai(\glx_k+2\qqq x)}.
\end{equation}
\end{theorem}
The  function $G(x)$ is plotted
in Figure \ref{FG}.

\begin{figure}[htbp]
	\centering
		\includegraphics[width=0.8\textwidth,angle=0]{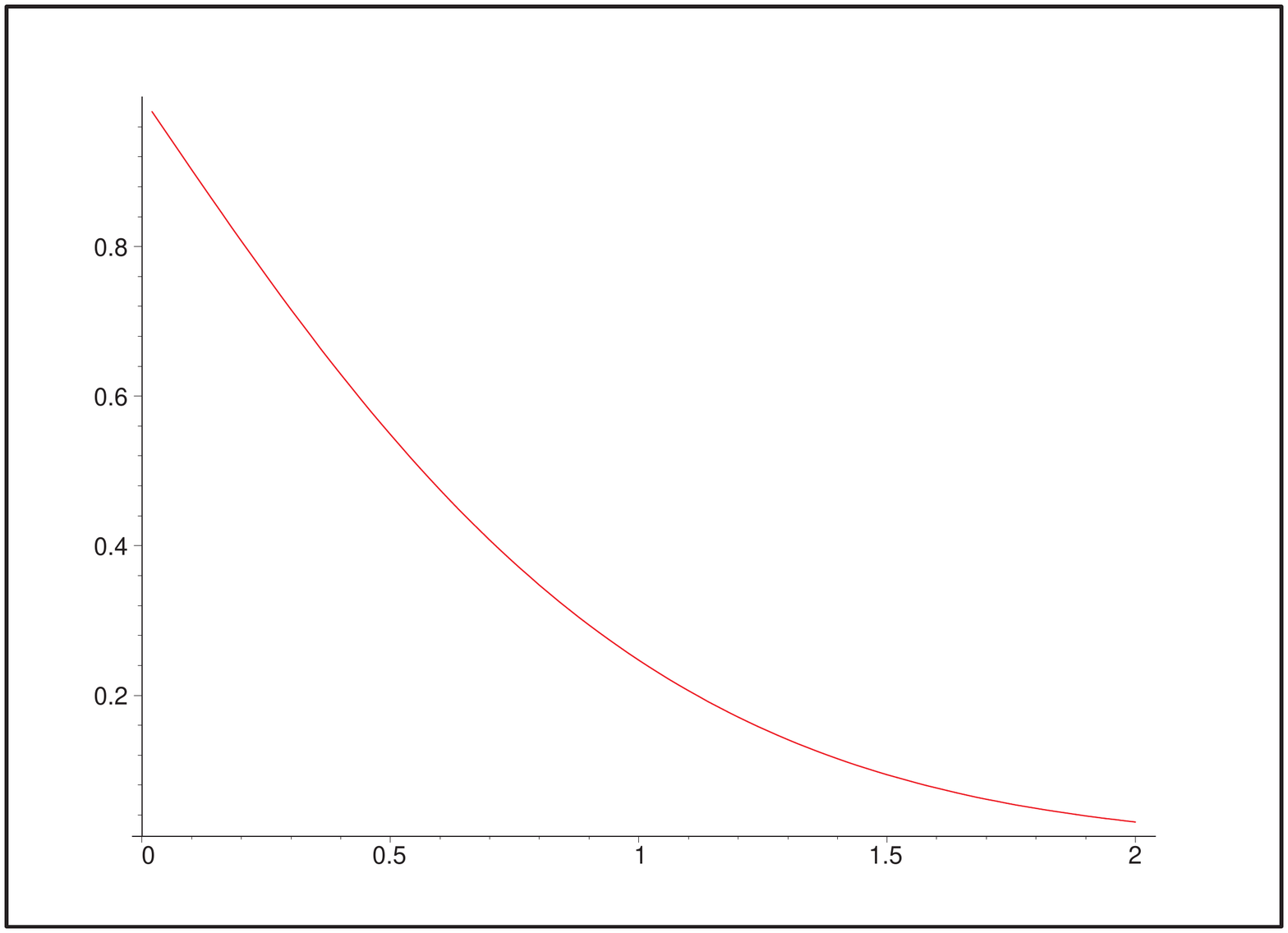}
	\caption{ $G(x)$}
	\label{FG}
\end{figure}

\begin{remark}\label{Rdiv}
  By  \eqref{Ai-} and \eqref{gl}, for any fixed $x>0$, 
$|\Ai(\glx_k+2\qqq x)|$ is usually of the order
$|\glx_k|\qqqqw\asymp k^{-1/6}$, and using also
\eqref{Ai'glk} and \eqref{Higlk}, the summands in \eqref{guy} are
(typically) of the order $k^{-1/6-1/6-2/3}=k\qw$, so this sum is
\emph{not} absolutely convergent. 
(For some values of $k$, the term may be smaller than
$k\qw$ because $\glx_k+2\qqq x$ may be close to
another zero, but such cases are infrequent and do not prevent the
series from being absolutely divergent.)

On the other hand, by \eqref{Hiz4}, $\pi\Hi(z)+ z\qw=O(|z|^{-4})$ on
the negative real axis, and it follows that the terms in
\eqref{guy1} are $O(k^{-3})$, so the series is absolutely convergent.
Moreover, 
since $\Ai$ is bounded on the real axis (see \eqref{Aioo} and \eqref{Ai-}),
the sum in \eqref{guy1} converges uniformly for $x\ge0$, and is thus a
continuous function of $x$; this is no surprise since we already have
remarked that $N$ has an absolutely continuous distribution, so $G$ is
continuous. Note also that \eqref{guy1} for $x=0$ is the trivial
$G(0)=1$, since each $\Ai(a_k)=0$, while \eqref{guy} does not hold for $x=0$.
\end{remark}

The sum in \eqref{guy1} can be differentiated termwise and we have
the following result, proved in \refS{SpfT1}.

\begin{theorem}  \label{T1d}
  $N$ and $M$ have absolutely continuous distributions with infinitely
  differentiable density functions, for $x>0$,
  \begin{align}
	\label{guy1'}
f_N(x)&=-2\qqq
\frac{\Ai'(2\qqq x)}{\Ai(0)}
-2\qqq \sumki
\frac{ \pi\Hi(\glx_k)+\glx_k\qw}
{\Ai'(\glx_k)} 
{\Ai'(\glx_k+2\qqq x)},
\\	\label{fmx}
f_M(x)&=2(1-G(x))f_N(x).
  \end{align}
\end{theorem}

Integral formulas for $f_N(x)$ will be given in \eqref{lj3} and \eqref{fny}.
The density functions  $f_N(x)$ and $f_M(x)$ are plotted
in Figures \ref{FfN} and~\ref{FfM}.

\begin{figure}[htbp]
	\centering
		\includegraphics[width=0.8\textwidth,angle=0]{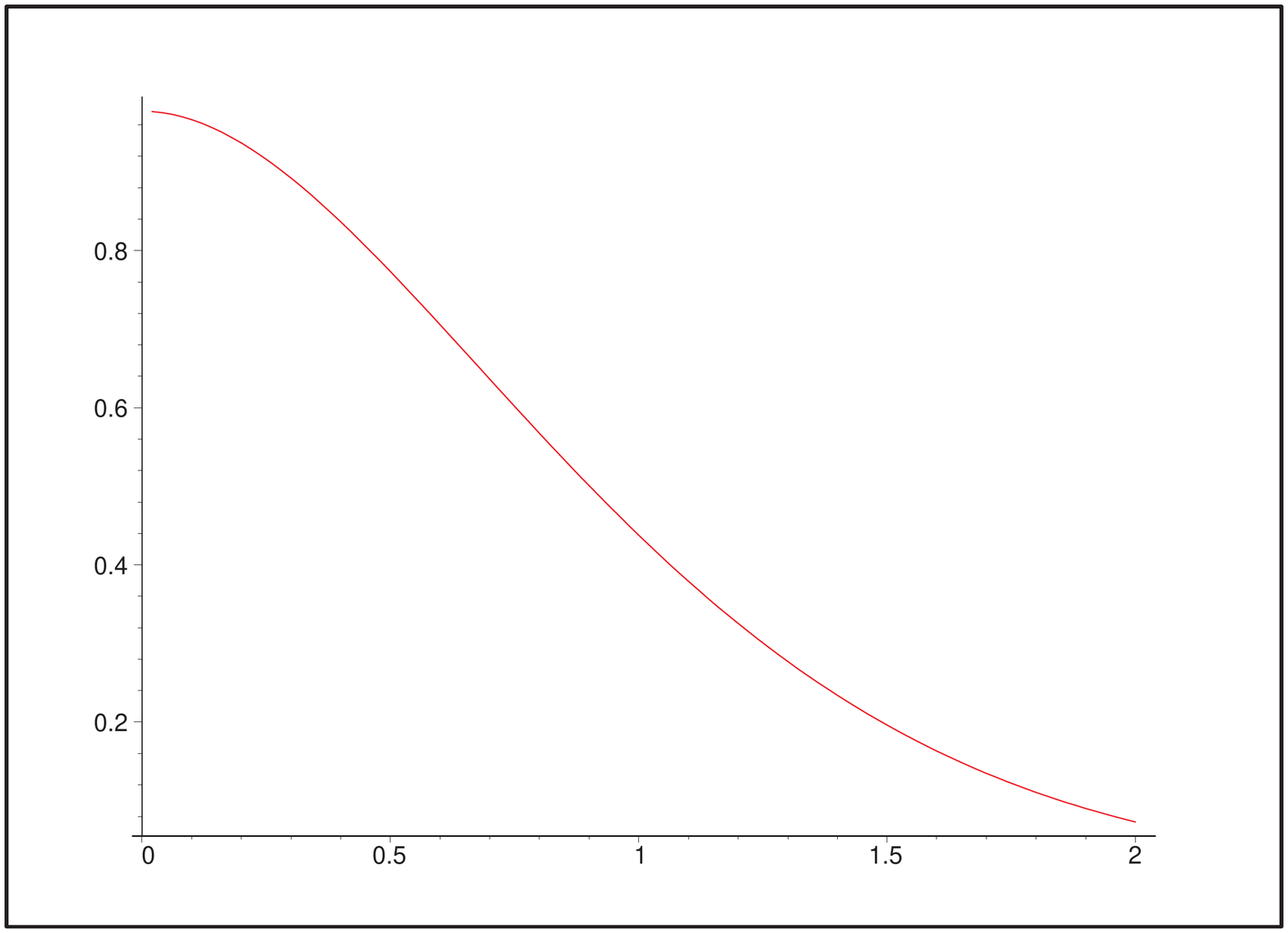}
	\caption{ $f_N(x)$}
	\label{FfN}
\end{figure}

\begin{figure}[htbp]
	\centering
		\includegraphics[width=0.8\textwidth,angle=0]{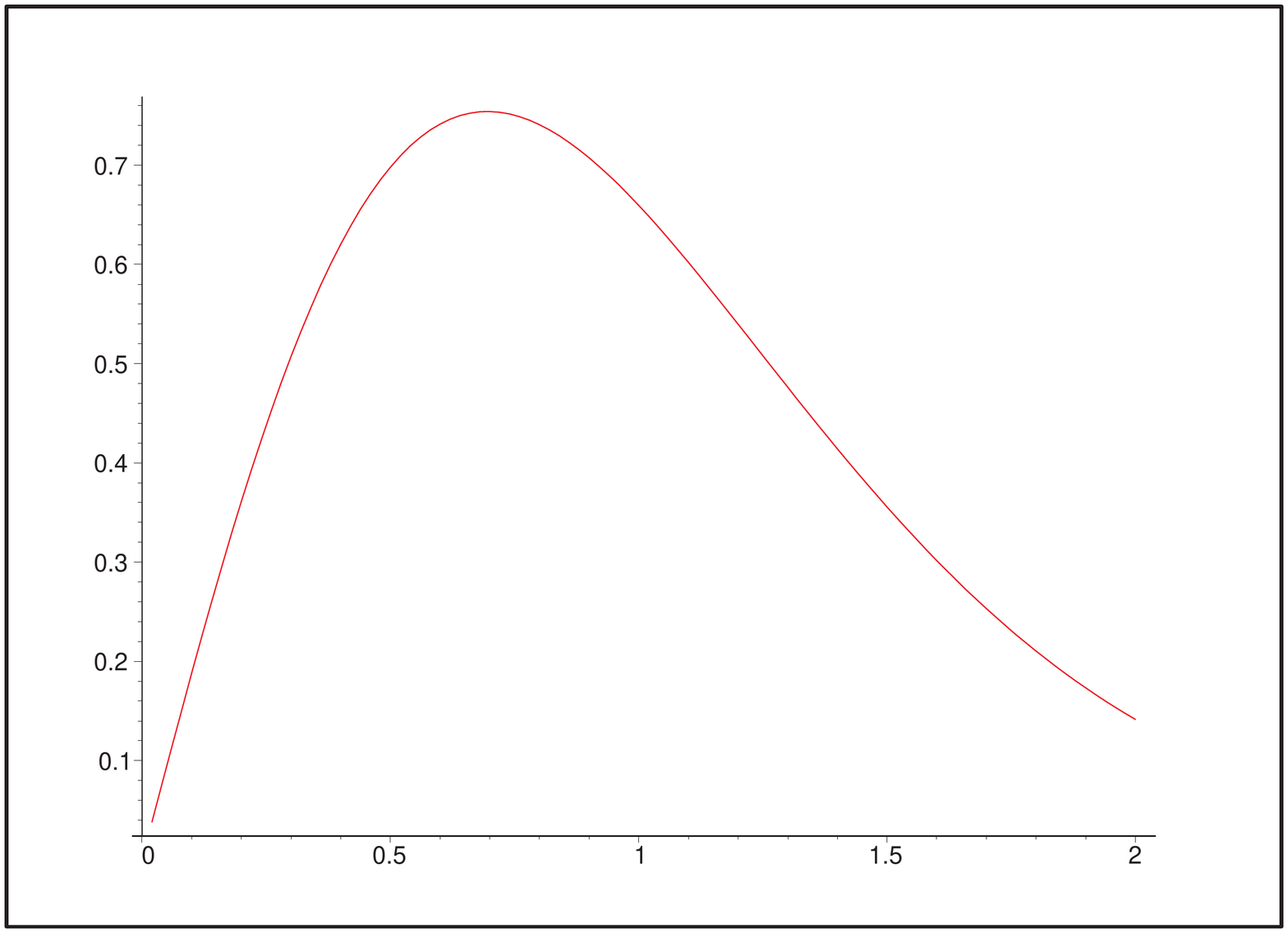}
	\caption{ $f_M(x)$}
	\label{FfM}
\end{figure}

\begin{remark}\label{Riv2}
In contrast,  the sum 
\begin{equation}\label{bad}
  2\qqq\pi
\sumki 
\frac{\Hi(\glx_k)}
{\Ai'(\glx_k)} 
{\Ai'(\glx_k+2\qqq x)}
\end{equation}
obtained by termwise differentiation of \eqref{guy} is \emph{not}
convergent for any $x\ge0$, as will be seen in \refS{SpfT1}.
\end{remark}

Moments of $M$ and $N$ now can be obtained from \eqref{enp} and
\eqref{emp} by integrating \eqref{guy1} termwise. This yields the
following result; see \refS{Smom} for proofs as well as related
integral formulas.
For higher moments, see \refR{Rp2}.

We define for convenience
\begin{equation}\label{Fik}
  \Fi(k)\=\pi\Hi(a_k)+a_k\qw.
\end{equation}
By \eqref{Hiz4}, $\Fi(k)=O(|a_k|^{-4})=O(k^{-8/3})$.

\begin{theorem}
  \label{TM}
The means and second moments of $M$ and $N$ are given by the
absolutely convergent sums
\begin{align}
  \E N 
&=\frac{1}{\dt 3 \Ai(0)}- \frac{\pi}{\dt} \sumki \Fi(k)\Gi(a_k) 
\label{en1}
\\
&=\frac{1}{\dt 3 \Ai(0)}-\frac{\pi}{\dt} 
  \sumki [\Fi(k)\Bi(a_k)-\Fi(k)\Hi(a_k)] ,
\label{en2}
\\
\E M 
&=   \frac{2\qqqb}{3 \Ai(0)} 
-\frac{\Ai'(0)^2}{\dt \Ai(0)^2} 
- \frac{1}{\dt} \sumki \lb 2\pi\Fi(k)\Bi(a_k)-\Fi(k)^2\rb
,
\label{E8}
\\
\E N^2 
&
=-\frac{\dt\Ai'(0)}{\Ai(0)}
+2\qqq\sumki\Fi(k)\bigsqpar{\pi a_k\Gi(a_k)-1}
\label{enn1}
\\&
=-\frac{\dt\Ai'(0)}{\Ai(0)}
+2\qqq\sumki\bigsqpar{\pi a_k\Fi(k)\Bi(a_k)-a_k\Fi(k)^2}
\label{enn2}
;
\\
\E M^2
&
=
-\frac{\dt5\,\Ai'(0)}{3\Ai(0)}
+2^{4/3}\sumki \Fi(k)
 \lb
\pi a_k\Bi(a_k)-\frac23a_k\Fi(k)+
\frac{2}{a_k^3}+\frac{2\Ai'(0)}{\Ai(0)a_k^2} \rb 
\non
\\
&\hskip8em
+2^{7/3}\sum_{k=1}^\II \sum_{j=1}^{k-1} 
 \frac{\Fi(k)\Fi(j)}{(a_k-a_j)^2}.
 \label{emm}
\end{align}
Numerically we have 
\begin{align*}
\E(N)&=0.69552 89995\ldots,
\\
\E (M) &=  0.99619 30199\ldots,
\\
\E(N^2)&=1.10279 82645\ldots,\\
\E (M^2) &=1.80329 57042\ldots,\\
\intertext{and thus}
\Var (N)&=\E(N^2)-\E(N)^2=0.61903 76754\ldots,\\
\Var (M)&=\E(M^2)-\E(M)^2=0.81089 51713\ldots.
\end{align*}
\end{theorem}

The numerical value for $\E M$ agrees with the one in \eqref{em0}.

Further formulas for moments are given in Theorems \refand{TJ1}{Tmean}.

\section{Distributions}\label{SpfT1}

Salminen \cite[Example 3.2]{Salminen} studied the hitting time
\begin{equation}\label{tau}
\tau\=\inf\set{t\ge 0:x+W(t)=-\gb t^2},  
\end{equation}
and
gave the formula \cite[(3.10)]{Salminen}, for $x,\gb>0$,
(with $\ga=-\gb$ in his notation),\begin{equation}\label{sal}
  f_\tau(t)=2\qqq \gb\qqqb
\sumki \exp\bigpar{2\qqq\gb\qqqb\glx_k t-\tfrac23\gb^2t^3}
\frac{\Ai(\glx_k+2\qqqb\gb\qqq x)}{\Ai'(\glx_k)}
\end{equation}
for the density function of $\tau$.
Note that $\tau$
is a defect random variable, and that $\tau=\infty$ if and
only if $\min_{t\ge0}(x+W(t)+\gb t^2)>0$. By
symmetry, $W\eqd -W$, and thus
\begin{equation*}
  \P(\tau=\infty)=\P\bigpar{\max_{t\ge0}(W(t)-\gb t^2-x<0)}
=\P(N_\gb<x).
\end{equation*}
Hence, choosing $\gb=1/2$,
\begin{equation}\label{sal2}
  \begin{split}
  G(x)&=
1-F_N(x)=\P(N\ge  x)=\P(\tau<\infty)=\intoo f_\tau(t)\dd t
\\&
=\intoo
2\qqqw
\sumki 
\frac{\Ai(\glx_k+2\qqq x)}{\Ai'(\glx_k)} 
\exp\bigpar{2\qqqw\glx_k t-\tfrac16t^3}
\dd t	
\\&
=\intoo
\sumki 
\frac{\Ai(\glx_k+2\qqq x)}{\Ai'(\glx_k)} 
\exp\bigpar{\glx_k t-\tfrac13t^3}
\dd t	.
  \end{split}
\end{equation}
If we formally integrate termwise we obtain \eqref{guy},
from \eqref{Hi}.
However, as seen in \refR{Rdiv}, the sum is not absolutely convergent,
so we cannot use \eg\ Fubini's theorem, and we have to justify the
termwise integration by a more complicated argument.

\begin{remark}
We have $|\glx_k|=\Theta(k\qqqb)$ by \eqref{gl}, 
and $|\Ai'(\glx_k)|=\Theta(k^{1/6})$ by \eqref{Ai'glk}; 
further, for fixed $x>0$,
$\Ai(\glx_k+2\qqq x)=O(|\glx_k|\qqqqw)=O(k^{-1/6})$ by \eqref{Ai-}.
Hence the sums in \eqref{sal} and \eqref{sal2} converge rapidly for each
fixed $t$, because of the negative term $\glx_k t$ in
the exponent. But the convergence rate is small for small $t$,
and when integrating we have the problem just described.   
\end{remark}

We begin by converting the sum in \eqref{sal} into a residue integral.
Fix $\gth_0\in(0,\pi/2)$ and $x_0\in(a_1,0)$, and let
$\gG=\gG(\gth_0,x_0)$ be the contour consisting of the ray
\set{re^{\ii(\pi+\gth_0)}} for $r$ from $\infty$ to $r_0\=|x_0|/\cos\gth_0$,
the line segment \set{x_0+\ii y} for
$y\in[-r_0\sin\gth_0,r_0\sin\gth_0]$
and the ray
\set{re^{\ii(\pi-\gth_0)}} for $r\in(r_0,\infty)$.

For (large) integers $N\in\bbN$, let $R_N\=(\frac32\pi N)\qqqb$, and
let $\gG_N\=\gG_N(\gth_0,x_0)$ be the closed contour obtained from
$\gG$ by cutting the infinite rays at $r=R_N$ and connecting them by
the arc $\gG_N'\=\set{R_Ne^{\ii\gth}:\gth\in[\pi-\gth_0,\pi+\gth_0]}$.

Note that by \eqref{gl}, $|\glx_N|<R_N<|\glx_{N+1}|$ (at least for
large $N$; in fact for all $N\ge1$). Thus, $\gG_N$ goes around the $N$
first zeros of $\Ai$; moreover, $\gG_N$ does not come too close to any
of the zeros; this is made more precise by the estimates in \refL{LA2}
and \refL{LA3}.

\begin{lemma}
  \label{LJ1}
Let $\tau=\tau_x$ be the hitting time \eqref{tau} for $\gb=1/2$ and
some $x>0$.
Then the defect random variable $\tau$ has the density function, for
$t>0$,
\begin{equation}
  \label{lj1}
f_\tau(t)
=\frac{1}{2\pi\ii}\intg 2\qqqw e^{2\qqqw zt-t^3/6}
\frac{\Ai(z+2\qqq x)}{\Ai(z)} \dd z.
\end{equation}
\end{lemma}

\begin{proof}
  For $x>0$ and $z\in\gG$ or $z\in\gG_N$, 
Lemmas \ref{LA} and \ref{LA2} yield
$\Ai(z+2\qqq x)/\Ai(z)=O(1)$,
and thus the integrand in
  \eqref{lj1}, $\Phi(z)$ say, is bounded by
  \begin{equation*}    
\Phi(z)=O\bigpar{\exp\bigpar{-2\qqqw t|\Re z|-t^3/6}}
=O\bigpar{\exp\bigpar{-2\qqqw t|\Re z|}}.
  \end{equation*}
This shows both that $\intg\Phi(z)\dd z$ is absolutely convergent, 
and that
\begin{equation}\label{hag2}
  \intgn\Phi(z)\dd z-\intg\Phi(z)\dd z\to0 \quad\text{as \Ntoo}.
\end{equation}

$\Phi(z)$ has simple poles at the zeros $\glx_k$ of $\Ai$, and
evaluating $\intgn\Phi(z)\dd z$ by residues, we see that it equals the
partial sum of the $N$ first terms of \eqref{sal}. Consequently,
\eqref{sal} yields
$  \intgn\Phi(z)\dd z\to f_\tau(t)$ as \Ntoo, and the result follows by
\eqref{hag2}. 
\end{proof}

\begin{remark}
  \label{RFT}
The contour $\gG$ in \eqref{lj1} can be deformed to the imaginary
axis, using the estimate in \refL{LA}. Hence, setting $z=2\qqq s\ii$,
\begin{equation}
  \label{ft}
e^{t^3/6} f_\tau(t)
=\frac1{2\pi}\intoooo e^{\ii st} \frac{\Ai(2\qqq(s\ii+x))}{\Ai(2\qqq s\ii)}
\dd s.
\end{equation}
Moreover, this holds also for $t\le0$, with $f_\tau(t)=0$, since the
\rhs{} of \eqref{ft} then easily is shown to vanish: writing it again as
a line integral along the imaginary axis we can move the line of
integration to $\Re z=\gs$, for any $\gs>a_1$, and for $t\ge0$ we may 
let $\gs\to+\infty$, again using \refL{LA}.

This exhibits $e^{t^3/6} f_\tau(t)$ as the inverse Fourier transform
of $s\mapsto \xfrac{\Ai(2\qqq( s\ii+x))}{\Ai(2\qqq s\ii)}$.
By \refL{LA}, this function is integrable and in $L^2$, and using
\eqref{Aioo} and \eqref{Aiioo}, it is seen that so is its derivative,
which implies that the Fourier transform $e^{t^3/6} f_\tau(t)$ is
integrable. The Fourier inversion formula yields
\begin{equation}  \label{ft2}
\intoooo e^{t^3/6} f_\tau(t) e^{-\ii st} \dd t
= \frac{\Ai(2\qqq(s\ii+x))}{\Ai(2\qqq s\ii)},
\qquad -\infty <s<\infty,
\end{equation}
and, more generally, by analytic continuation,
\begin{equation}\label{lt}
\intoooo e^{t^3/6} f_\tau(t) e^{-zt} \dd t
= \frac{\Ai(2\qqq(z+x))}{\Ai(2\qqq z)},
\qquad \Re z \ge0
\end{equation}
(and, in fact, for $\Re z > a_1$).
This formula for the Laplace transform is (in a more general version)
given by \citetq{Theorem 2.1}{Groeneboom}, where
$e^{t^3/6} f_\tau(t)$ is denoted $h_{1/2,x}(t)$; see also \eqref{v2} below.
Conversely, this formula from \cite{Groeneboom} yields by Fourier
inversion \eqref{ft} and \eqref{lj1}, so we could have used it instead
of \eqref{sal} from \cite{Salminen} as our starting
point. 
We give an alternative proof of \eqref{lt} in \refApp{Aalt}, which
thus gives us a self-contained proof of \refL{LJ1}.
(\citet{Groeneboom}, \citet{Salminen} and our \refApp{Aalt}
use similar methods. See
also \refApp{AD} for another approach.)
\end{remark}

\begin{remark}
  For our purposes we consider only $x>0$ in \eqref{tau}. For $x<0$, 
  the hitting time is \as{} finite; its distribution is found in \citet{ML98}.
\end{remark}

\begin{lemma}
  \label{LJ2}
For $x>0$,
\begin{equation}
  \label{lj2}
G(x)
=\frac{1}{2\ii}\intg
\frac{\Hi(z)\Ai(z+2\qqq x)}{\Ai(z)} \dd z.
\end{equation}
\end{lemma}

\begin{proof}
  We have $G(x)=\P(\tau<\infty)=\intoo f_\tau(t)\dd t$.
Now integrate \eqref{lj1} with respect to $t$ and interchange the
order of integration, which is allowed by Fubini's theorem and the
estimate 
\refL{LA}, which implies that for $z\in\gG$, the integrand in
\eqref{lj1} is bounded by $O\bigpar{\exp(-cx|z|\qq-t^3/6)}$.
The result follows by \eqref{Hi} (and a change of variables $t=2\qqq t_1$).
\end{proof}

The integral in \eqref{lj2} is absolutely convergent and converges
rapidly for any fixed $x>0$ by 
\refL{LA} and \eqref{Hiz}.
We denote the integrand in \eqref{lj2} by $\Psi(z)=\Psi(z;x)$.

\begin{proof}[Proof of \refT{T1}]
Fix $x>0$.
By \eqref{Hiz} and Lemmas \refand{LA}{LA2}, for $z=R_Ne^{\ii(\gf+\pi)}$ with
$|\gf|\le\frac\pi2$, 
$\Psi(z)=O\bigpar{|z|\qw\exp\bigpar{-x|z|\qq|\gf|/\pi}}$.
It follows that $\intgn\Psi(z)\dd z-\intg\Psi(z)\dd z\to0$ as \Ntoo,
and thus $G(z)=\frac{1}{2\ii}\lim_{N\to\infty}\intgn\Psi(z)\dd z$.
Evaluating $\intgn$ by residues, we obtain \eqref{guy}.  

To obtain \eqref{guy1}, we take out the first term in the expansion
\eqref{Hizk} of $\Hi(z)$, and write \eqref{lj2} as
\begin{equation*}
G(x)
=
\frac{1}{2\ii}\intg
\frac{(\Hi(z)+\pi\qw z\qw)\Ai(z+2\qqq x)}{\Ai(z)} \dd z
-\frac{1}{2\ii\pi}\intg
\frac{\Ai(z+2\qqq x)}{z\Ai(z)} \dd z.
\end{equation*}
The first integral can be converted to a sum of residues by the
argument just given for \eqref{lj2}, which yields the sum in
\eqref{guy1}. Indeed, we have better estimates now, and the resulting
sum is absolutely convergent as seen in \refR{Rdiv}.

For the second integral we instead close the contour on the right, by
a large circular arc \set{Re^{\ii t}} for $t$ from $\pi-\gth_0$ to 
$-(\pi-\gth_0)$; it follows by \refL{LA} that the error tends to 0 as
$R\to\infty$. Inside this closed contour, $\Ai$ has no zeros, so the
only pole is at $z=0$ where the residue is $\Ai(2\qqq x)/\Ai(0)$.
The result follows, noting that we go around this contour in the
negative direction.
\end{proof}

\begin{remark}\label{Rhigher}
  We may also use the expansion \eqref{Hizk} of $\Hi$ with more terms.
In general, subtracting the sum with $L$ terms in \eqref{Hizk} from $\Hi$
in \eqref{lj2}
yields an integral that can be converted to a sum of residues as
above; this sum is similar to the ones in \eqref{guy} and
\eqref{guy1}, and the terms are now of order $k^{-1-2L}$.
We also have the integral with the subtracted terms; this is a linear
combination of terms of the type $\intg z^{-n-1}\Ai(z+2\qqq x)/\Ai(z)$,
which as above equals $-2\pi\ii$ times the residue at 0, so this
integral can be written as a combination of derivatives of $\Ai$ at
$2\qqq x$ and 0; by the equation $\Ai''(z)=z\Ai(z)$, and successive
derivations of this equation, the result can be written as
$p_1(x)\Ai(2\qqq x)+p_2(x)\Ai'(2\qqq x)$ for some polynomials $p_1$
and $p_2$ (depending on $L$), whose coefficients are rational
functions in $\Ai(0)$ and $\Ai'(0)$.
We leave the details to the reader.
\end{remark}

\begin{proof}[Proof of \refT{T1d}]
If $|\arg z|<\pi-\gd$ and $|z|\ge1$, say, then by \refL{LA3},
$|\Ai'(z)/\Ai(z)|\asymp |z|\qq$. More generally, using also
$\Ai''(z)=z\Ai(z)$, and differentiating this equation further, 
by induction, $\Ai^{(m)}(z)/\Ai(z)=O(|z|^{m/2})$ for every fixed $m\ge0$. 

It follows by \refL{LA} and \eqref{Hiz} that 
for every fixed $m\ge0$, $x$ in a fixed interval $(x_0,x_1)$ with
$0<x_0<x_1$, and 
$z\in\gG$, 
\begin{equation*}
  \frac{\partial^m}{\partial x^m} \Psi(z;x) 
=
O\lrpar{\zz\qw|z+x|^{m/2}e^{-cx\zz\qq}}
=
O\lrpar{\zz^{m/2-1}e^{-cx_0\zz\qq}}.
\end{equation*}
Consequently, we can differentiate \eqref{lj2} under the integral sign
an arbitrary number of times; this shows that $G$ is infinitely
differentiable on $(0,\infty)$ and that $N$ has an infinitely
differentiable density $f_N=-G'$ given by
\begin{equation}
  \label{lj3}
f_N(x)=-G'(x)
=\frac{2\qqq\ii}{2}\intg
\frac{\Hi(z)\Ai'(z+2\qqq x)}{\Ai(z)} \dd z.
\end{equation}

This integral can be evaluated as a sum of residues as for
\eqref{guy1} in the proof of \refT{T1} above, by first adding $\pi\qw
z\qw$ to $\Hi(z)$, which yields \eqref{guy1'}.
Alternatively, and perhaps simpler, 
\eqref{Ai'-} and \eqref{Ai'glk} imply that, uniformly for $0\le x\le x_0$
for any fixed $x_0>0$, 
$\Ai'(a_k+2\qqq x)=O(|\glx_k|\qqqq)=O(\Ai'(\glx_k))$, and thus the
terms in the sum in \eqref{guy1'} are $O(|\pi\Hi(\glx_k)+\glx_k\qw|)
=O(|\glx_k|^{-4})=O(k^{-8/3})$. Hence we can integrate the sum in
\eqref{guy1'} termwise; equivalently, we can differentiate
\eqref{guy1} termwise, which yields \eqref{guy1'}.

The result for $M$ and \eqref{fmx} follow from $F_M(x)=F_N(x)^2=(1-G(x))^2$.
\end{proof}

To see that the sum \eqref{bad} does not converge for any $x>0$, let
$y\=2\qqq x$. Take $x=|a_k|-y$ in \eqref{Ai'-}. Since
then, by Taylor's formula and \eqref{gl},
\begin{equation*}
  \tfrac23(|a_k|-y)\qqc
=   \tfrac23|a_k|\qqc -y|a_k|\qq+o(1)
=  \frac{\pi(4k-1)}{4} -y(3\pi k/2)\qqq+o(1),
\end{equation*}
we obtain
\begin{equation*}
  \Ai'(a_k+y)=-\pi\qqw|a_k|\qqqq\bigpar{\cos(\pi k-y(3\pi
  k/2)\qqq+o(1))+o(1)}
\end{equation*}
and thus 
\begin{equation*}
\frac{  \Ai'(a_k+y)}{\Ai'(a_k)}=
\frac{\cos(\pi k-y(3\pi k/2)\qqq)+o(1)}{\cos(\pi k)+o(1)}
={\cos(y(3\pi k/2)\qqq)+o(1)}.
\end{equation*}
Let $I_n$ be the interval $[c(2\pi n)^3,c(2\pi n+1)^3]$, with 
$c\=2y^{-3}/(3\pi)$. For $k\in I_n$, we have 
$y(3\pi k/2)\qqq\in [2\pi n,2\pi n+1]$ and thus 
$\xfrac{  \Ai'(a_k+y)}{\Ai'(a_k)}\ge \cos 1+o(1)>0.5$, if $n$ is large.
Further, by \eqref{Hi-} and \eqref{gl}, 
$\Hi(a_k)\sim\pi\qw|a_k|\qw\sim\pi\qw(3\pi/2)\qqqbw k\qqqbw$.
Hence the term in \eqref{bad}, $t_k$ say, satifies, for some constants
$c_1,c_2>0$ and $k\in I_n$,
\begin{equation*}
  t_k \ge c_1 k\qqqbw \ge c_2 n\qww.
\end{equation*}
Since there are $\Theta(n^2)$ integers in $I_n$, the sum over them is
$\Theta(1)$, and thus the sum in \eqref{bad} diverges.
(The case $x=0$ is simple.)

\section{Moments}\label{Smom}

\begin{lemma}\label{Llp}
  For every fixed $p\ge1$,  uniformly in $a\ge0$,
  \begin{align}
\intoo x^{p-1} \abs{\Ai(x-a)}\dd x = O(a^{p-1/4}+1)\label{lem1},
\\
\intoo x^{p-1} \abs{\Ai(x-a)}^2\dd x = O(a^{p-1/2}+1)\label{lem2}.
  \end{align}
\end{lemma}

\begin{proof}
  For $0\le x< a$ we have
$x^{p-1}\le a^{p-1}$ and $\abs{\Ai(x-a)}=O((a-x)\qqqqw)$ by
  \eqref{Ai-},
so   
\begin{equation*}
 \int_0^a x^{p-1} \abs{\Ai(x-a)}\dd x  
=O\lrpar{a^{p-1}\int_0^a(a-x)\qqqqw\dd x}
=O\bigpar{a^{p-1} a^{3/4}},
\end{equation*}
and similarly 
$ \int_0^a x^{p-1} \abs{\Ai(x-a)}^2\dd x  
=O\bigpar{a^{p-1} a^{1/2}}$.
For larger $x$ we use the rapid decrease in \eqref{Aioo}, which implies
\begin{equation*}
  \int_a^\infty x^{p-1} \abs{\Ai(x-a)}\dd x 
= 
  \int_0^\infty (y+a)^{p-1} \abs{\Ai(y)}\dd y 
=O(1+a^{p-1}),
\end{equation*}
and similarly for $\int_a^\infty x^{p-1} \abs{\Ai(x-a)}^2\dd x $.
The result follows.
\end{proof}

\begin{proof}[Proof of \refT{TM}]
Write \eqref{guy1} as
\begin{equation}
  \label{gsum}
G(2\qqqw x) = \sumk \gfc(k)\Ai(x+a_k),
\end{equation}
where we for convenience define $a_0=0$ and 
\begin{align*}
  \gfc(0)&\=\frac1{\Ai(0)},
\\
\gfc(k)&\=\frac{\pi\Hi(a_k)+a_k\qw}{\Ai'(a_k)}
=\frac{\Fi(k)}{\Ai'(a_k)}
, \qquad k\ge1.
\end{align*}
By \eqref{Hiz4}, \eqref{Ai'glk} and \eqref{gl},
\begin{equation}\label{17}
  \abs{\gfc(k)}=O\bigpar{\abs{a_k}^{-4-1/4}}
=O\bigpar{k^{-17/6}},
\qquad k\ge1.
\end{equation}
By \eqref{lem1}, 
$\norm{\Ai(x+a_k)}_{L^1((0,\infty),\dd x)}=O(\abs{a_k}^{3/4})=O(k\qq)$, 
$k\ge1$, 
and thus the sum in \eqref{gsum} converges absolutely in
$L^1((0,\infty),\dd x)$, 
so it may be integrated termwise. 
Consequently, using  \eqref{enp} and \eqref{aAI},
\begin{equation*}
\E N 
= \intoo G(x)\dd x 
= 2\qqqw\intoo G(2\qqqw x)\dd x 
=2\qqqw\sumk\gfc(k)\AI(a_k).
\end{equation*}
We have $\AI(0)=\intoo\Ai(x)\dd x=1/3$ \cite[10.4.82]{AS}, see \eqref{aAI},
and, for $k\ge1$, $\AI(a_k)=-\pi\Ai'(a_k)\Gi(a_k)$ by \eqref{AIG} since
$\Ai(a_k)=0$. Thus \eqref{en1} follows, and so does \eqref{en2} by \eqref{GBH}.

\refL{Llp} and \eqref{17}
similarly imply that
\eqref{gsum} converges absolutely also in $L^2((0,\infty),\dd x)$.
Hence, the integral $\intoo G(x)^2\dd x$ can be obtained by termwise
integration in \eqref{gsum},
using \eqref{IAi2}, \eqref{IAi2k},
\eqref{IAi2abkl}, \eqref{Ioo0k}: 
\begin{align*}
\intoo G(x)^2\dd x 
&= 2\qqqw\intoo G(2\qqqw x)^2\dd x 
\\&
=2\qqqw\sumk\sumell\gfc(k)\gfc(\ell)\intoo\Ai(x+a_k)\Ai(x+a_\ell)\dd x
\\&
=2\qqqw\gfc(0)^2(\Ai'(0))^2
-2\qqqw2\sumki\gfc(0)\gfc(k)\frac{\Ai(0)\Ai'(a_k)}{a_k}
\\ &\hskip12em
+2\qqqw\sumki\gfc(k)^2 (\Ai'(a_k))^2
\\&
=2\qqqw\parfrac{\Ai'(0)}{\Ai(0)}^2
-2\qqqb\sumki\frac{\Fi(k)}{a_k}
+2\qqqw\sumki\Fi(k)^2.
\end{align*}

Thus, by \eqref{emp} and \eqref{en2}, 
\begin{align*}
  \E M&=2\intoo G(x)\dd x-\intoo G(x)^2\dd x=2\E N-\intoo G(x)^2\dd x
\\
&=
\frac{2\qqqb}{3 \Ai(0)}
-2\qqqw\parfrac{\Ai'(0)}{\Ai(0)}^2
-\frac{1}{\dt}   \sumki \Fi(k)\lb2\pi\Bi(a_k)-2\pi\Hi(a_k) 
-\frac{2}{a_k}+\Fi(k)\rb,
\end{align*}
and \eqref{E8} follows by the definition of $\Fi(k)$.

Similarly,
\eqref{gsum} converges absolutely in $L^1((0,\infty),x\dd x)$ 
and $L^2((0,\infty),x\dd x)$ too, and
termwise integration in
\eqref{gsum} yields,
using \eqref{IxAikx}, \eqref{AIG}, 
\eqref{IxAi2}, \eqref{Ioox0k}, \eqref{IxAi2kx},
\eqref{IxAi2abkl},
\begin{align*}
\intoo G(x)x\dd x 
&= 2\qqqbw\intoo G(2\qqqw x)x\dd x 
\\&
=2\qqqbw\sumk\intoo  \gfc(k)\Ai(x+a_k)x\dd x 
\\&
=2\qqqbw\sumk\gfc(k)\bigpar{-\Ai'(a_k)-a_k\AI(a_k)}
\\&
=-\frac{\Ai'(0)}{\dtc\Ai(0)}
+2\qqqbw\sumki\Fi(k)\bigpar{-1+\pi a_k\Gi(a_k)};
\\
\intoo G(x)^2x\dd x 
&= 2\qqqbw\intoo G(2\qqqw x)^2x\dd x 
\\&
=2\qqqbw\sumk\sumj\gfc(k)\gfc(j)\intoo x\Ai(x+a_k)\Ai(x+a_j)\dd x
\\&
=-\frac{\Ai'(0)}{\dtc 3\Ai(0)}
+\frac{2}{\dtc}\sumki \Fi(k)
 \lb-\frac{2}{a_k^3}-\frac{2\Ai'(0)}{\Ai(0)a_k^2} \rb \\
&\qquad+\frac{1}{\dtc}\lb -\sumki \Fi(k)^2\frac23 a_k
-2\sum_{k=1}^\II \sum_{j=1}^\II\kl k\neq j \kr
 \frac{\Fi(k)\Fi(j)}{(a_k-a_j)^2}\rb.
\end{align*}
By \eqref{enp}, $\E N^2=2\intoo G(x)x\dd x$, and
\eqref{enn1}--\eqref{enn2} follow.
Similarly, 
$\E M^2=4\intoo G(x)x\dd x-2\intoo G(x)^2x\dd x=2\E N^2-2\intoo G(x)^2x\dd x
$, and \eqref{emm} follows.
The numerical evaluation is done by \maple, using the method discussed
in \refS{Snum}.
\end{proof}

\begin{remark}
The formula above for $\intoo G(x)^2\dd x$ may be simplified. In fact,   
$(\Ai'(0)/\Ai(0))^2=\sumki a_k\qww$ by \eqref{Aiparseval}, and thus 
the formula can be written
\begin{equation}
  \label{Gparseval}
\intoo G(x)^2\dd x = 
2\qqqw \sumki(\Fi(k)-1/a_k)^2
=2\qqqw\pi^2\sumki \Hi(a_k)^2.
\end{equation}
This can be seen as an instance of Parseval's formula, see
\refR{RSturm}. However, although simpler than our expression above,
this sum converges more slowly and is less suitable for our purposes.
\end{remark}

\begin{remark}
  As is well-known, see \cite[10.4.4--5]{AS},
$\Ai(0)=3\qqqbw/\Gamma(2/3)=3^{-1/6}\Gamma(1/3)/(2\pi)$
and
$\Ai'(0)=-3\qqqw/\Gamma(1/3)=-3^{1/6}\Gamma(2/3)/(2\pi)$.
We prefer to keep $\Ai(0)$ and $\Ai'(0)$ in our formulas.
\end{remark}

\begin{remark}
  \label{Rp2}
Higher moments can be computed by the same method, with Airy integrals
evaluated as shown in \refApp{AiryB}, but in order to get convergence,
one may have to use a version of \eqref{guy1} with more terms taken
out of the expansion \eqref{Hizk} of $\Hi$, as discussed in \refR{Rhigher}.
We do not pursue the details.
\end{remark}

We can also give integral formulas based on \refL{LJ2}.

\begin{theorem}
  \label{TJ1}
The moments of $M$ and $N$ are given by, for any real $p>0$,
\begin{multline*}
\E N^p = -p 2^{-p/3-1} i \intg \intoo x^{p-1}\Ai(z+x)\dd x\,
\frac{\Hi(z)}{\Ai(z)}\dd z,
\\
\shoveleft{\E M^p = -p 2^{-p/3} i \intg \intoo x^{p-1}\Ai(z+x)\dd x\,
\frac{\Hi(z)}{\Ai(z)}\dd z}	
\\
+
p 2^{-p/3-2} \intg\intg \intoo x^{p-1}\Ai(z+x)\Ai(w+x)
\dd x\,
\frac{\Hi(z)}{\Ai(z)}\frac{\Hi(w)}{\Ai(w)}\dd z	\dd w.
\end{multline*}
\end{theorem}
\begin{proof}
  Immediate from \eqref{enp}--\eqref{emp} and \refL{LJ2} (with a change of
  variables $x\to2\qqqw x$).
The double and triple integrals converge absolutely by \refL{LA} and
  \eqref{Hiz}. 
\end{proof}

For integer $p$, the integrals over $x$ in \refT{TJ1} can be
evaluated by  the formulas in \refApp{AiryB}.
In particular, by \eqref{IAi} and \eqref{Iooab},
\begin{multline}
  \label{emc2}
\E M = -2\qqqw i \intg \frac{\AI(z)\Hi(z)}{\Ai(z)}\dd z
\\
+2^{-7/3}\intg\intg\frac{\Ai(z)\Ai'(w)-\Ai'(z)\Ai(w)}{z-w}
\,\frac{\Hi(z)\Hi(w)}{\Ai(z)\Ai(w)}\dd z\dd w.
\end{multline}
Although there is no singularity when $z=w$ in the double integral in
\eqref{emc2}, it may be advantageous to use different, disjoint,
contours for $z$ and $w$. 
Remember that $\gG=\gG(\gth_0,x_0)$.
We choose $\gth_1\in(\gth_0,\pi/2)$ and $x_1\in(x_0,0)$, and let
$\gGG\=\gG(\gth_1,x_1)$. Then $\gG$ and $\gGG$ are disjoint;
moreover, if $z\in\gG$ and $w\in\gGG$, then 
\begin{equation}
\label{dzw}
|z-w|\ge c\max(\abs z,\abs w)   
\end{equation}
for some $c>0$.
Furthermore, \eqref{dzw} holds also if $z\in \gG_n$ and $w\in\gG'_M$
with $M\ge 2N$. We can replace the double integrals $\intg\intg$ in
\refT{TJ1} and \eqref{emc2} by $\intg\intgg$.
Taking residues, this leads to another formula with sums over Airy zeros.

\begin{theorem}
  \label{Tmean}
  \begin{align}
\label{tmeann}
\E N
= 2\qqqw\pi^2 \sumki\Hi(a_k)(\Hi(a_k)-\Bi(a_k))
,	
\\\label{tmeanm}
\E M
= 2\qqqw\pi^2 \sumki\Hi(a_k)(\Hi(a_k)-2\Bi(a_k))
.	
  \end{align}
\end{theorem}

These formulas are closely related to \eqref{en2}--\eqref{E8}. They
are simpler, but less 
suitable for numerical calculations since they do not even converge absolutely;
the terms in the sums decrease as $k^{-5/6}$ by \eqref{Biglk} and
\eqref{Higlk}. 
(However, they alternate in sign, and the sums converge.)
The formulas \eqref{tmeann} and \eqref{tmeanm}
are what we obtain if we substitute \eqref{guy} in \eqref{enp} and \eqref{emp}
(with $p=1$) and integrate termwise; however, since the resulting sums are not
absolutely convergent, termwise integration has to be justified
carefully, and we use a detour via complex integration.

\begin{proof}
  Let 
  \begin{equation*}
Q(z,w)\=\intoo\Ai(z+x)\Ai(w+x)\dd x.	
  \end{equation*}
The integral converges absolutely by \eqref{Aioo} for all complex $z$
and $w$, uniformly in compact sets, and thus $Q$ is an entire function
of two variables; moreover, \eqref{IAi2} and \eqref{IAi2ab} yield the
explicit formulas
\begin{align}
  Q(z,z)&=\Ai'(z)^2-z\Ai(z)^2, \label{q1}\\
Q(z,w)&=\frac{\Ai(z)\Ai'(w)-\Ai'(z)\Ai(w)}{z-w}, \qquad z\neq w.
\label{q2}
\end{align}
By \refT{TJ1} and \eqref{emc2}, using  $\gGG$ as discussed above, 
\begin{align*}
	\E N
&= -2^{-4/3} i \intg \frac{\AI(z)\Hi(z)}{\Ai(z)}\dd z,
\\
	\E M
&= -2\qqqw i \intg \frac{\AI(z)\Hi(z)}{\Ai(z)}\dd z
+2^{-7/3}\intg\intgg\frac{Q(z,w) \Hi(z)\Hi(w)}{\Ai(z)\Ai(w)}\dd z\dd w.
\end{align*}

First consider the simple integral, $\intg\Phi(z)\dd z$ say.
It follows from \refL{LA3} and \eqref{Hiz} that $\intg\Phi(z)\dd
z-\int_{\gG_n}\Phi(z)\to0$ as  \ntoo, and we find by the residue theorem 
applied to $\gG_n$, letting \ntoo, together with
\eqref{AIG} and \eqref{GBH},
\begin{equation*}
  \begin{split}
  \intg \frac{\AI(z)\Hi(z)}{\Ai(z)}\dd z
&=2\pi\ii \sumki\frac{\AI(a_k)\Hi(a_k)}{\Ai'(a_k)}
\\&
=-2\pi^2\ii \sumki\Gi(a_k)\Hi(a_k)
\\&
=-2\pi^2\ii \sumki\Bi(a_k)\Hi(a_k)
+2\pi^2\ii \sumki\Hi(a_k)^2
.	
  \end{split}
\end{equation*}
The sums converge, \eg{} by the argument just given, but only the
final sum $\sum\Hi(a_k)^2$ converges absolutely, by
\eqref{Biglk}--\eqref{Giglk}. This yields the result
\eqref{tmeann} for $N$.

Next consider the double integral. If $z\in\gG$ and $w\in\gG'$, or
$z\in \gG_n$ and $w\in \gGG_m$ with $m\ge 2n$, then by \eqref{dzw},
\eqref{q2} and \refL{LA3},
\begin{equation*}
  \begin{split}
  \lrabs{\frac{Q(z,w)}{\Ai(z)\Ai(w)}}
&
\le \frac{C}{\abs z+\abs w}
\lrpar{\Bigabs{\frac{\Ai'(z)}{\Ai(z)}}+\Bigabs{\frac{\Ai'(w)}{\Ai(w)}}}
\\&
\le \frac{C}{\abs z+\abs w}\bigpar{|z|\qq+|w|\qq}
\le {C}|z|\qqqqw|w|\qqqqw.	
  \end{split}
\end{equation*}
It follows that $\intg\intgg-\int_{\gG_n}\int_{\gGG_m}\to0$ 
as $m\ge 2n\to\infty$.
Using the residue theorem for first $\int_{\gG_n}$ and then
$\int_{\gGG_m}$, with $m=2n$, we find that the double integral above
equals
\begin{equation*}
  \begin{split}
\lim_{\ntoo} (2\pi i)^2
\sum_{j=1}^n\sum_{k=1}^{2n}
\frac{Q(a_j,a_k)\Hi(a_j)\Hi(a_k)}{\Ai'(a_j)\Ai'(a_k)}.
  \end{split}
\end{equation*}
By \eqref{q2}, $Q(a_j,a_k)=0$ when $j\neq k$, and
$Q(a_k,a_k)=\Ai'(a_k)^2$ by \eqref{q1}. Hence the double integral equals
$ -4\pi^2 \sumki \Hi(a_k)^2$.

The result \eqref{tmeanm} follows by combining the terms.
\end{proof}

\section{Proof of \refT{TEM}}\label{SpfTEM}

\begin{proof}[Proof of \eqref{em1}]
We shall use \citet{Groeneboom}.     
Fix $\gamc>0$. 
(We may choose \eg{} $\gamc=1/2$ as in other parts of this paper
by \eqref{scale}, but for ease of comparison with
\cite{Groeneboom}, and because we find it instructive to see
how the homogeneity works, we write the proof for a general $\gamc$.)
Fix also $s\ge0$ and define 
\begin{equation}
  \vs\=\max_{t\ge-s}(W(t)-\gamc t^2). 
\end{equation}
Thus, $\vx0=N_\gamc$, 
and $\vs\upto\vx\infty=M_\gamc$ as
$s\to\infty$.

For $t\ge -s$, the process $\ws(t)\=W(t)-W(-s)$
is a Brownian motion, starting at $0$ at time $-s$.
Define
\begin{equation}\label{v0}
\tM=\tM_s\=\max_{t\ge-s} \bigpar{\ws(t)-\gamc t^2}
  =\vs-W(-s),
\end{equation}
and let $\tau=\tau_s$ be the (\as{} unique) time
with $\ws(\tau)-\gamc\tau^2=\tM$. 
Note that $\tM\ge\ws(-s)-\gamc s^2=-\gamc s^2$ and
$\tau\ge-s$ (strict inequalities hold a.s.)
\citetq{Corollary 3.1}{Groeneboom} applies to $W_s(t)-\gamc t^2$ (with $s$
replaced by $-s$ and $x=-\gamc s^2$),
and shows that $\tau$ and $\tM$ have a joint
density, for $t>-s$ and $y>x=-\gamc s^2$,
\begin{equation}\label{v1}
  \begin{split}
\ftm(t,y)
&=
\exp\bigpar{-\tfrac23\gamc^2(t^3+s^3)+2\gamc s(y+\gamc s^2)}\hcyc(t+s)\kc(t)	
\\&
=
\exp\bigpar{\tfrac43\gamc^2s^3+2\gamc sy}\hcyc(t+s)\gc(t)	.
  \end{split}
\end{equation}
where the functions $\hcyc$, $\kc$ and $\gc$ are given in
\cite{Groeneboom}.
Integrating over $t\ge-s$ we find, with
$h_{\gamc,a}(t)=0$ for $t<0$ and $\cgc(t)\=\gc(-t)$,
the density $\fm$ of $\tM$ as, for $y>-\gamc s^2$,
\begin{align}
  \fm(y)
&=\int_{-s}^\infty f(t,y)\dd t
= \exp\bigpar{\tfrac43\gamc^2s^3+2\gamc sy}\int_{-s}^\infty\hcyc(t+s)\gc(t)\dd t
\nonumber
\\
&= \exp\bigpar{\tfrac43\gamc^2s^3+2\gamc sy}\hcyc*\cgc(s).	
\label{fm}  
\end{align}

By \cite[Theorem 2.1]{Groeneboom}, for $a>0$,
$\hca\ge0$ and $\hca$ has the Laplace transform,
for $\gl>0$, (see also \refR{RFT})
\begin{equation}
  \label{v2}
\intoo e^{-\gl u} \hca(u)\dd u =
  \frac{\Ai\bigpar{\fcqqq a+\xi}}{\Ai(\xi)},
\quad \xi\=\bccqqqw \gl.
\end{equation}
Letting $\gl\downto0$ we see that
\begin{equation*}
  \intoo \hca(u)\dd u = \Ai\bigpar{\fcqqq a}/\Ai(0)<\infty,
\end{equation*}
so $\hca\in L^1(\bbR)$ and \eqref{v2} holds
for all complex $\gl$ with $\Re\gl\ge0$ by
analytic continuation. In particular, $\hca$ has the Fourier
transform, see \eqref{ft2},
\begin{equation*}
  \FT \hca(\go) =\intoooo e^{-\ii\go u}\hca(u)\dd u
=
  \frac{\Ai\bigpar{\fcqqq a+\ii\bccqqqw\go}}{\Ai\bigpar{\ii\bccqqqw\go}},
\qquad \go\in\bbR.
\end{equation*}
Furthermore, by \eqref{v1}, 
$\gc\ge0$, 
and by \cite[Corollary 3.1]{Groeneboom}, 
$\gc\in L^1(\bbR)$ with the Fourier transform
\begin{equation}\label{ftgc}
  \FT{\gc}(\go)
=\intoooo e^{-\ii\go u}\gc(u)\dd u
=
  \frac{2\qqq \gamc\qqqw}{\Ai\bigpar{-\ii\bccqqqw\go}},
\qquad \go\in\bbR.
\end{equation}
Hence, for $y>-\gamc s^2$, $\hcyc*\cgc$ has the
Fourier transform
\begin{equation}\label{v2a}
  \FT \hcyc(\go) \FT \gc(-\go)
=
2\qqq \gamc\qqqw
\frac{\Ai\bigpar{\fcqqq (y+\gamc s^2)+\ii\bccqqqw\go}}
{\Ai\bigpar{\ii\bccqqqw\go}^2}.  
\end{equation}
Since $|\Ai(\ii y)|\asymp|y|\qqqqw\exp\bigpar{\tfrac{\sqrt2}{3}|y|\qqc}$
as $y\to\pm\infty$ by \eqref{Aiy},
\eqref{ftgc} implies 
$|\FT\gc(\go)|\asymp|\go|\qqqq\exp\bigpar{-\tfrac{1}{3\gamc}|\go|\qqc}$;
consequently $\FT\gc\in L^1(\bbR)$.
Furthermore, $\FT\hca(\go)$ is bounded for each $a$
(by $\hca\in L^1$ or by \refL{LA}), and thus the product 
$\FT\hcyc(\go)\FT\gc(-\go)\in L^1(\bbR)$. 
Consequently, the Fourier inversion formula applies to
\eqref{v2a}, and \eqref{fm} thus yields, for $y>-\gamc s^2$,
\begin{equation}\label{fm2}
  \begin{split}
\fm(y)
&= \exp\bigpar{\tfrac43\gamc^2s^3+2\gamc sy}\frac{2\qqq \gamc\qqqw}{2\pi}
\\&\hskip8em
\times
\intoooo e^{\ii s\go}
\frac{\Ai\bigpar{\fcqqq	(y+\gamc s^2)+\ii\bccqqqw\go}}{\Ai(\ii\bccqqqw\go)^2}
\dd\go
\\&
=\frac{\fcqqq}{2\pi} \exp\bigpar{\tfrac43\gamc^2s^3+2\gamc sy}
\intoooo e^{\ii s\bccqqq v}
\frac{\Ai\bigpar{\fcqqq	(y+\gamc s^2)+\ii v}}{\Ai(\ii v)^2}
\dd v
.  
  \end{split}
\end{equation}

Multiplying by $e^{zy}$ and integrating, we obtain for any
$z\in\bbC$ the
following, where the double integral is absolutely convergent by \refL{LA},
\begin{equation}\label{jukk}
  \begin{split}
\E e^{z\tM}
&
=\int_{-\gamc s^2}^\infty e^{zy}\fm(y)\dd y
=\int_{0}^\infty
e^{zx-z\gamc s^2}\fm(x-\gamc s^2)\dd x
\\&
=\frac{\fcqqq}{2\pi} 
\int_{x=0}^\infty\int_{v=-\infty}^\infty 
e^{\frac43\gamc^2s^3+2\gamc sx-2\gamc^2s^3+zx-\gamc s^2z+\ii s\bccqqq v}
\frac{\Ai\bigpar{\fcqqq	x+\ii v}}{\Ai(\ii v)^2}
\dd v \dd x
\\&
=\frac{1}{2\pi} e^{-\frac23\gamc^2s^3-\gamc s^2z}
\int_{v=-\infty}^\infty  \int_{x=0}^\infty
e^{(z+2\gamc s)\fcqqqw x+\ii s\bccqqq v}
\frac{\Ai\bigpar{x+\ii v}}{\Ai(\ii v)^2}
\dd x \dd v
\\&
=\frac{1}{2\pi} e^{-\frac23\gamc^2s^3-\gamc s^2z}
\int_{v=-\infty}^\infty  
\frac{e^{\ii s\bccqqq v}}{\Ai(\ii v)^2}
\int_{x=0}^\infty
e^{(z+2\gamc s)\fcqqqw x}
{\Ai\bigpar{x+\ii v}}
\dd x \dd v.
\end{split}
\end{equation}

Since $\Ai$ is bounded on the negative real axis by
\eqref{Ai-}, \refL{LV} implies that, for $\Re
z>0$,
\begin{equation*}
  \begin{split}
  \intoo e^{zt}\Ai(t) \dd t
&=e^{t^3/3}-  \int_{-\infty}^0
e^{zt}\Ai(t)\dd t
=e^{t^3/3}+O\biggpar{ \int_{-\infty}^0
e^{\Re zt}\dd t}
\\&
=e^{t^3/3}+O\Bigparfrac1{\Re z}.	
  \end{split}
\end{equation*}
Moreover, \eqref{Aiy} implies that for
$y\in\bbR$ and $z\in\bbC$, 
\begin{equation*}
  \int_0^{iy} e^{zt}\Ai(w) \dd w
=O\Bigpar{e^{\frac{\sqrt2}3|y|\qqc+|\Im z||y|}}.
\end{equation*}

Hence, if $\Re z\ge \fcqqqw$, say, and $v\in\bbR$, then, 
using Cauchy's integral formula on a large rectangle with vertices
$\set{0,\ii y, R, R+\ii y}$ and letting
$R\to\infty$, using \eqref{Aioo} to control the tails,
\begin{equation*}
  \begin{split}
\intoo e^{zx}\Ai(x+iv)\dd x 	
&=\int_{\ii v}^{\infty+ïi v} e^{zw-\ii vz}\Ai(w)\dd w
\\&
=e^{-\ii vz}\intoo e^{zw}\Ai(w)\dd w 	
-
e^{-\ii vz}\int_0^{\ii v} e^{zw}\Ai(w)\dd w 	
\\&
=e^{z^3/3-\ii vz}
+O\Bigpar{e^{\frac{\sqrt2}3|v|\qqc+|\Im z||v|}}.
  \end{split}
\end{equation*}
Hence, if $\Re z+2\gamc s>1$, \eqref{jukk} yields, using
\eqref{Aioo} again for the error term,
\begin{equation*}
  \begin{split}
\E e^{z\tM}
&
=\frac{1}{2\pi} e^{-\frac23\gamc^2s^3-\gamc s^2z}
\int_{-\infty}^\infty  
\frac{e^{\ii s\bccqqq v}}{\Ai(\ii v)^2}
e^{(z+2\gamc s)^3/(12\gamc)-\ii v (z+2\gamc s)\fcqqqw}
\dd v
\\&
\qquad{}
+
O\biggpar{
e^{-\frac23\gamc^2s^3-\gamc s^2\Re z}
\int_{-\infty}^\infty  
\frac{1}{|\Ai(\ii v)|^2}
{e^{\frac{\sqrt2}3|v|\qqc+|\Im z\fcqqqw||v|}}
\dd v
}
\\&
=\frac{1}{2\pi} 
e^{\frac{z^3}{12\gamc}+\frac{z^2s}2}
\int_{-\infty}^\infty  
\frac{e^{-\ii \fcqqqw zv}}{\Ai(\ii v)^2}
\dd v
+
O\bigpar{
e^{-\frac23\gamc^2s^3-\gamc s^2\Re z+O(|\Im z|^2)}
}
\end{split}
\end{equation*}
In particular, if $2\gamc s>2$, say, this holds uniformly for
$|z|<1$, and we may differentiate the analytic functions at
$z=0$ and obtain
\begin{equation*}
  \begin{split}
\E \tM
&
=\frac{1}{2\pi} 
\int_{-\infty}^\infty  
\frac{-\ii \fcqqqw v}{\Ai(\ii v)^2}
\dd v
+
O\bigpar{
e^{-\frac23\gamc^2s^3+\gamc s^2}
}.
\end{split}
\end{equation*}
Since $\E \vs=\E(\tM+W(-s))=\E\tM$ by \eqref{v0},
  we find by letting $s\to\infty$ and choosing $\gamc=1/2$,
  \begin{equation*}
\E M=\lim_{s\to\infty} \E \vs	
=\frac{2\qqqw}{2\pi\ii} 
\int_{-\infty}^\infty  
\frac{v}{\Ai(\ii v)^2}
\dd v,
  \end{equation*}
which is \eqref{em1}.
\end{proof}

\begin{remark}
  Setting $s=0$ and $\gamc=1/2$ in \eqref{fm2}, we obtain another formula for
  the density of $N$: For $y>0$, 
\begin{equation}\label{fny}
  \begin{split}
f_N(y)
=\frac{2\qqq}{2\pi} 
\intoooo 
\frac{\Ai\bigpar{2\qqq	y+\ii v}}{\Ai(\ii v)^2}
\dd v
.  
  \end{split}
\end{equation}
By residue calculus, 
as with similar integrals in \eg\ the proof of
\refT{T1}, this may be written as a sum
of residues 
of $2\qqq \Ai(2\qqq y+z)/\Ai(z)^2$
at the poles $a_k$; however, now the poles are double and
we omit the details.
\end{remark}

The integral formulas \eqref{em1}--\eqref{em2} can
be transformed into each other by properties of the Airy functions as follows.

\begin{proof}[Proof of \eqref{em2}]
The integrand in \eqref{em1} is analytic except at the zeros of
$\Ai$, which lie on the negative real axis.
Furthermore,
by \eqref{Aioo}, $|\Ai(z)|$ is exponentially large,
so the integrand in \eqref{em1} is exponentially small,
as $\abs z\to\infty$ with
$\pi/3+\gd\le|\arg(z)|\le\pi-\gd$; 
in particular when 
$\pi/2\le|\arg(z)|\le2\pi/3$. 
Hence, we can
deform the integration path from the imaginary axis to any reasonable
path in this domain.
We choose to integrate along the two rays from the origin with
$\arg z=\pm2\pi/3$, and obtain thus
\begin{equation}\label{ema}
 \E M=
-\frac{2\qqqw}{2\pi\ii}\int_{-e^{-2\pi\ii/3}\infty}^{e^{2\pi\ii/3}\infty} 
\frac{z\dd z}{\Ai(z)^2}
=
-\frac{2\qqqw}{2\pi\ii}\int_{0}^{\infty} 
\sum_{\pm}\pm\frac{e^{\pm4\pi\ii/3}t\dd t}{\Ai(e^{\pm2\pi\ii/3}t)^2},
\end{equation}
which by the formula \cite[10.4.9]{AS} 
\begin{equation}\label{A2pi/3}
  \Ai(ze^{\pm2\pi\ii/3})=\tfrac12e^{\pm\pi\ii/3}\bigpar{\Ai(z)\mp\ii\Bi(z)}
\end{equation}
yields
\begin{equation*}
  \begin{split}
 \E M&=
-\frac{2\qqqw}{2\pi\ii}\int_{0}^{\infty} 
\sum_{\pm}\pm4\frac{e^{\pm2\pi\ii/3}t\dd t}{\bigpar{\Ai(t)\mp\ii\Bi(t)}^2}
\\&=
-\frac{2\qqqb}{\pi\ii}\int_{0}^{\infty} 
\sum_{\pm}\pm\frac{e^{\pm2\pi\ii/3}\bigpar{\Ai(t)\pm\ii\Bi(t)}^2}
{\bigpar{\Ai(t)^2+\Bi(t)^2}^2}\,t\dd t
\\&=
-\frac{2\qqqb}{\pi\ii}\int_{0}^{\infty} 
\sum_{\pm}\pm\frac{(-1\pm\sqrt3\ii)
  \bigpar{\Ai(t)^2-\Bi(t)^2\pm2\ii\Ai(t)\Bi(t)}}
{2\bigpar{\Ai(t)^2+\Bi(t)^2}^2}\,t\dd t
\\&=
-\frac{2\qqqb}{\pi}\int_{0}^{\infty} 
\frac{{\sqrt3\Ai(t)^2-\sqrt3\Bi(t)^2-2\Ai(t)\Bi(t)}}
{\bigpar{\Ai(t)^2+\Bi(t)^2}^2}\,t\dd t,
  \end{split}
\end{equation*}
which proves \eqref{em2}.  
\end{proof}

\begin{proof}[Proof of \eqref{em3}]
To prove \eqref{em3}, we as above
deform the integration path in \eqref{em1} 
and integrate along the two rays from the origin with
$\arg z=\pm2\pi/3$ and obtain \eqref{ema}.
We now use the indefinite integral 
\begin{equation}\label{alb1}
 \int\frac{\dd z}{\Ai(z)^2} = \pi\frac{\Bi(z)}{\Ai(z)}
\end{equation}
given by  \cite{AlbrightG86} (and easily verified by
derivation, using the Wronskian
$\Ai(z)\Bi'(z)-\Ai'(z)\Bi(z)=1/\pi$
\cite[10.4.10]{AS}). 
We have by \eqref{A2pi/3} and \AS{10.4.6}
\begin{equation*}
  \Bi(ze^{2\pi\ii/3})
=e^{\pi\ii/6}  \Ai(ze^{-2\pi\ii/3})
+e^{-\pi\ii/6}  \Ai(z)
=\tfrac12e^{-\pi\ii/6}\bigpar{3\Ai(z)+\ii\Bi(z)}
\end{equation*}
and thus, by \eqref{A2pi/3} again,
\begin{equation*}
\frac{\Bi(ze^{2\pi\ii/3})}{\Ai(ze^{2\pi\ii/3})}
=e^{-\pi\ii/2}\frac{3\Ai(z)+\ii\Bi(z)}{\Ai(z)-\ii\Bi(z)}
=\ii-4\ii\frac{\Ai(z)}{\Ai(z)-\ii\Bi(z)}.
\end{equation*}
By \eqref{Aioo} and \eqref{Bioo}, this converges
rapidly to $\ii$ as $z\to\infty$ along the
positive real axis. Hence an integration by parts yields
\begin{align}
\int_{0}^{e^{2\pi\ii/3}\infty} \frac{z}{\Ai(z)^2}\dd z  
&
=
\Bigsqpar{\pi z\Bigpar{\frac{\Bi(z)}{\Ai(z)}-\ii}\dd z}_0^{e^{2\pi\ii/3}\infty}
-\int_{0}^{e^{2\pi\ii/3}\infty}\pi\Bigpar{\frac{\Bi(z)}{\Ai(z)}-\ii}\dd z  
\nonumber\\&
=
0+
e^{2\pi\ii/3}\pi\int_{0}^{\infty}4\ii\frac{\Ai(t)}{\Ai(t)-\ii\Bi(t)}\dd t .
\label{emc}\\&
=
2\pi\int_{0}^{\infty}(-\ii-\sqrt3)
\frac{\Ai(t)^2+\ii\Ai(t)\Bi(t)}{\Ai(t)^2+\Bi(t)^2}\dd t .
\label{emb}
\end{align}
The integral along the line from
$e^{-2\pi\ii/3}\infty$ to 0 equals $-1$ times the complex
conjugate of the integral in \eqref{emb}, so we obtain from
\eqref{ema} and \eqref{emb}, 
\begin{equation*}
 \E M=
-\frac{2\qqqw}{\pi}\Im\int_{0}^{e^{2\pi\ii/3}\infty} 
\frac{z\dd z}{\Ai(z)^2}
=
2\qqqb\int_{0}^{\infty}
\frac{\Ai(t)^2+\sqrt3\Ai(t)\Bi(t)}{\Ai(t)^2+\Bi(t)^2}\dd t ,
\end{equation*}
which is \eqref{em3}.
\end{proof}

\begin{proof}[Proof of \eqref{em4}]
Follows similarly by
\eqref{ema} and \eqref{emc}, we omit the details.
\end{proof}

\section{Numerical computation}\label{Snum}

The sums in \refT{TM} converge, but rather slowly. For example, in
\eqref{en1}, 
$\abs{\Fi(k)\Gi(k)}=\abs{\Fi(k)}\,\abs{\Bi(k)-\Hi(k)}
\sim c k^{-17/6}$ for some $c>0$, see the asymptotic expansions below.

To obtain numerical values with high accuracy
of the sums in \eqref{en1}--\eqref{emm}, we therefore use
asymptotic expansions of the summands.
More precisely, for each sum, we first compute $\sum_{k=1}^{199}$
numerically.
(All computations are done by \maple.)
We then evaluate $\sum_{200}^\infty$ by replacing the summands by the
first terms in the following asymptotic expansions.
(We provide here only one or two terms in each asymptotic expansions, but we
use at least $5$ terms in our numerical computations.) 
We note that $\Bi(a_k)$ alternates in sign, so 
for terms containing it,
we group the terms with
$k=2j$ and $k=2j+1$ together, for $j\ge100$. 
The resulting infinite sums are readily computed numerically.
(They can be expressed in the Riemann zeta function at some
points.)

We use the expansions, see \AS{(10.4.94), (10.4.105)}, \eqref{Hizk},
\eqref{Fik}, \eqref{Bioo-}, 
\begin{align*}
  |a_k|&
\sim \frac{3^{2/3}\pi^{2/3}}{2\qqqb} \Bigpar{k^{2/3}-\frac16 k\qqqw+\dots}
\\ 
\Hi(a_k)&\sim-\frac1\pi a_k\qw -\frac2\pi a_k^{-4} + \dots
=\frac{\dtc}{ 3^{2/3}\pi^{5/3}}\Bigpar{k^{-2/3}+\frac16 k^{-5/3}+\dots}\\
\Fi(k)&\sim-2a_k^{-4}+\dots
=-\frac{2^{11/3}}{ 3^{8/3}\pi^{8/3}}k^{-8/3}+\ldots,\\
\Bi(a_{k})&\sim(-1)^k\frac{2^{1/6}}{3^{1/6}\pi^{2/3}}k^{-1/6}+\ldots,\\
g_1(j)&:=
\Fi(2j)\Bi(a_{2j})+\Fi(2j+1)\Bi(a_{2j+1})
\sim-\frac{17\cdot 3^{1/6}}{162\,\pi^{10/3}}j^{-23/6}+\ldots,
\\
g_2(j)&:=
a_{2j}\Fi(2j)\Bi(a_{2j})+a_{2j+1}\Fi(2j+1)\Bi(a_{2j+1})
\sim\frac{13\cdot 3^{5/6}}{162\,\pi^{8/3}}j^{-19/6}+\ldots.
\end{align*}

Thus, for example, $\E N= \EN_0+\EN_1+\EN_2$, where
\begin{align*}
\EN_0&=0.6955290109\ldots ,\\  
\EN_1&:=-\frac{\pi}{\dt} \sumkcc \Fi(k)\Bi(a_k)
=-\frac{\pi}{\dt} \sumjc g_1(j)
=0.5317
\lcdots 10^{-8},
\\
\EN_2&:=\frac{\pi}{\dt} \sumkcc \Fi(k)\Hi(a_k)
=-0.16722
\lcdots 10^{-7},
\end{align*}
yielding $\E(N)=0.6955289995\ldots$.
Similarly, 
$\E M= \EM_0+\EM_1+\EM_2$, with
\begin{align*}
\EM_0&=0.9961930092\ldots,\\
\EM_1&:=-\frac{\pi}{\dt} \sumkcc 2\Fi(k)\Bi(a_k)
=2\EN_1
= 0.10635
\lcdots 10^{-7},
\\
\EM_2&:=2\qqqw \sumkcc \Fi(k)^2
=0.20
\lcdots 10^{-13},
\end{align*}
yielding $\E M=0.9961930199\ldots$, which   
fits with \eqref{em0}.
Actually, we have $15$ digits of precision with the expansions we have used.

For the second moments, we compute the sums in \eqref{enn2} and \eqref{emm}
in the same way. 
For the double sum in \eqref{emm}, we compute the sum with $k\le199$ exactly,
and find using asymptotics the tail sum 
\begin{equation*}
  2^{7/3}\sum_{k=200}^\II \sum_{j=1}^{k-1} 
 \frac{\Fi(k)\Fi(j)}{(a_k-a_j)^2}
=0.4\lcdots 10^{-10}.
\end{equation*}
(We may, for example, use the first term asymptotics for $\Fi(j)$ and $a_j$ 
for $j\ge200$, since the sum is small and only a low relative
precision is needed.) 

The (complementary) distribution function $G(x)$ and the density
function $f_N(x)$ may, for any given $x>0$,  be computed to high
precision from \eqref{guy1} and \eqref{guy1'}
in the same way; for the tails of the sums 
$\sumkzz\Fi(k)\Ai(a_k+2\qqq x)/\Ai'(a_k)$ and
$\sumkzz\Fi(k)\Ai'(a_k+2\qqq x)/\Ai'(a_k)$ we use the asymptotic
expansion of $\Fi(k)$ given above together with, see 
\AS{(10.4.94), (10.4.97), (10.4.105)} and \eqref{Ai-}--\eqref{Ai'-}, 
\begin{align*}
  \Ai(a_k+2\qqq x)&\sim (-1)^{k+1}
 \frac{2^{1/6}}{3^{1/6}\pi^{2/3}}
{\sin\bigpar{(3\pi)\qqq x k\qqq}}\,  k^{-1/6}+\dots
\\
  \Ai'(a_k+2\qqq x)&\sim
(-1)^{k+1}
 \frac{3^{1/6}}
{2^{1/6}\pi^{1/3}}  
{\cos\bigpar{(3\pi)\qqq x k\qqq}}\,
k^{1/6}+\dots
\intertext{for fixed $x$, and, as a special case,}
  \Ai'(a_k)&\sim (-1)^{k+1}
 \frac{3^{1/6}}{2^{1/6}\pi^{1/3}}  k^{1/6}+\dots
\end{align*}
The distribution and density functions of $M$ then are given by
\eqref{Fm}--\eqref{Gm} and \eqref{fmx}.

As an illustration, we plot the tail sum (from $k=200$) for 
$f_N(x)$ in 
\refF{FcorfN}.

\begin{figure}[htbp]
	\centering
		\includegraphics[width=0.8\textwidth,angle=0]{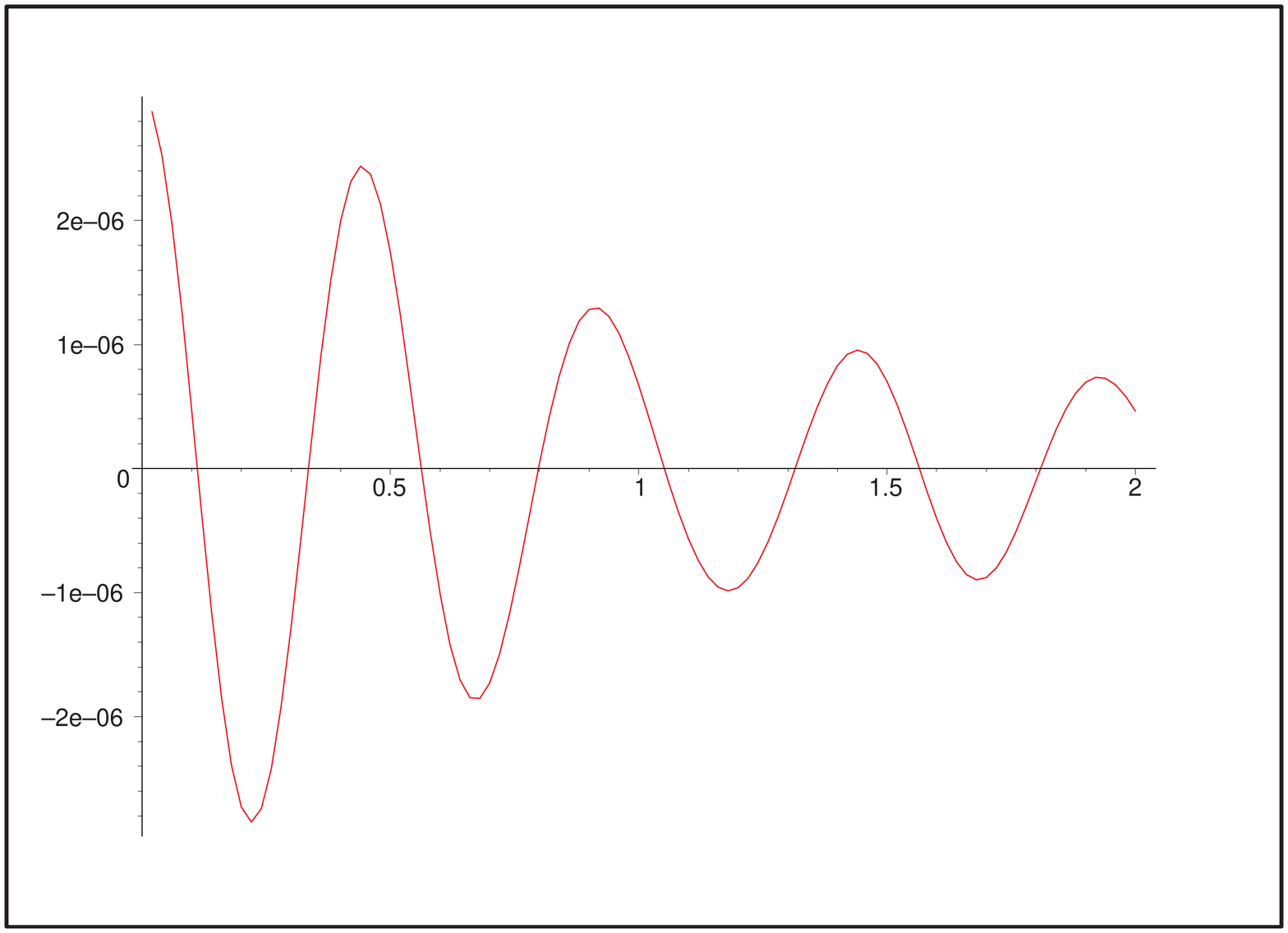}
	\caption{Tail sum from $k=200$ for $f_N(x)$}
	\label{FcorfN}
\end{figure}

\appendix

\section{Some Airy function estimates} \label{AiryA}

The Airy function $\Ai(x)$ and its derivative have along the positive real
axis, and more generally as $\abs z\to\infty$ with
$|\arg z|<\pi-\gd$ for any fixed $\gd>0$, asymptotic
expansions given in full in
\cite[10.4.59 and 10.4.61]{AS}; we need only the leading terms: 
\begin{align}
  \Ai(z)
&\sim \tfrac1{2\sqrt \pi} z\qqqqw e^{-\frac23z\qqc}, 
\label{Aioo}
\\
  \Ai'(z)
&\sim -\tfrac1{2\sqrt \pi} z\qqqq e^{-\frac23z\qqc}
.
\label{Aiioo}
\end{align}
In particular, along the imaginary axis, for $y\in\bbR$,
\begin{equation}\label{Aiy}
|\Ai(\ii y)|
\sim \tfrac1{2\sqrt \pi} |y|\qqqqw e^{\frac{\sqrt2}3|y|\qqc}.
\end{equation}

Along the negative real axis, $\Ai$ and $\Ai'$
oscillate and have zeros; 
by \AS{(10.4.60), (10.4.62)} (where further terms are given),
we have the asymptotic formulas
\begin{align}\label{Ai-}
\Ai(-x)&=\pi\qqw  x\qqqqw\bigpar{\sin\bigpar{\tfrac23x\qqc+\tfrac{\pi}4}+o(1)}
=O(x\qqqqw),
\\
\Ai'(-x)&=-\pi\qqw  x\qqqq\bigpar{\cos\bigpar{\tfrac23x\qqc+\tfrac{\pi}4}+o(1)}
\label{Ai'-}.
\end{align}
and 
more generally, 
as $\abs z\to\infty$ in any domain 
$|\arg z|<\frac23\pi-\gd$, 
\begin{multline}\label{Ai-z}
  \Ai(-z)
= \pi\qqw z\qqqqw
\biggpar{\sin\Bigpar{\frac23z\qqc+\frac\pi4}\bigpar{1+O(|z|\qqcw)}
+O\bigpar{|z|\qqcw}
},
\end{multline}
\begin{multline}\label{Ai'-z}
  \Ai'(-z)
= -\pi\qqw z\qqqq
\biggpar{\cos\Bigpar{\frac23z\qqc+\frac\pi4}\bigpar{1+O(|z|\qqcw)}
+O\bigpar{|z|\qqcw}}.
\end{multline}

For the companion Airy function $\Bi(x)$,
we use the following estimate \cite[10.4.63]{AS},
valid along the positive real
axis, and more generally as $\abs z\to\infty$ with
$|\arg z|<\pi/3-\gd$,
\begin{align}
  \Bi(z)
&\sim  \pi\qqw z\qqqqw e^{\frac23z\qqc}. 
\label{Bioo}
\end{align}
On the negative axis, we have \cite[10.4.64]{AS},
\begin{align}
  \Bi(-x)
&=  
\pi\qqw  x\qqqqw\bigpar{\cos\bigpar{\tfrac23x\qqc+\tfrac{\pi}4}+o(1)}.
\label{Bioo-}
\end{align}


\begin{lemma}
  \label{LA}
For every $\gd>0$ there exists $c=c(\gd)$ such that
if $|\arg z|\le\pi-\gd$ and $x\ge0$, then
\begin{equation*}
  {\frac{\Ai(z+x)}{\Ai(z)}} = O\Bigpar{e^{-cx|z|\qq-cx\qqc}}.
\end{equation*}
\end{lemma}

\begin{proof}
We may assume that $0<\gd<\pi/2$.
It follows from \eqref{Aioo} that for $|\arg z|\le\pi-\gd$,
\begin{equation}\label{er1}
  |\Ai(z)|\asymp(1+|z|)\qqqqw\exp\bigpar{-\tfrac23\Re (z\qqc)}.
\end{equation}
Since $|\arg(z+x)|\le|\arg(z)|\le\pi-\gd$,
\eqref{er1} can be used for $z+x$ too.
If $\Re z\le0$, then $|z+x|\ge|\Im z|\ge(\sin\gd)|z|$, and if $\Re z\ge0$, then
$|z+x|\ge|z|$. Hence, \eqref{er1} implies
\begin{equation*}
  \lrabs{\frac{\Ai(z+x)}{\Ai(z)}} 
= O\Bigpar{e^{-\frac23\Re ((z+x)\qqc)+\frac23\Re (z\qqc)}}.
\end{equation*}
Further,
\begin{equation*}
  \begin{split}
  \frac23\Re ((z+x)\qqc)-\frac23\Re (z\qqc)
&=\Re\int_0^x(z+t)\qq\dd t 
\\&
\ge\cos\frac{\pi-\gd}2 \int_0^x|z+t|\qq\dd t.
  \end{split}
\end{equation*}
Moreover, $|z+t|\ge c_1(|z|+t)$ by elementary geometry (use,
\eg, the sine theorem on the triangle with vertices
$0,z,-t$);
hence $|z+t|\qq\ge c_2(|z|+t)\qq$, and the result follows.
\end{proof}

Recall that $R_N\=\bigpar{\frac32\pi N}\qqqb$.

\begin{lemma}
  \label{LA2}
Let $x\ge0$ be fixed. Assume that $z=R_N\eith$ with
$\gth\in[\frac\pi2,\frac{3\pi}2]$ and $N\ge1$.
Then (with implicit constants depending on $x$ but not on $N$ or
$\gth$),
\begin{equation}\label{la2i}
  {\frac{\Ai(z+x)}{\Ai(z)}}
= O\Bigpar{e^{-x R_n\qq|\gth-\pi|/\pi}}
\end{equation}
and
\begin{equation}\label{la2ii}
 \lrabs{\Ai(z)}
\asymp|z|\qqqqw e^{\lrabs{\Im\frac23(-z)\qqc}}
= R_N\qqqqw e^{\frac23 R_n\qqc\sin(3|\gf|/2)}.
\end{equation}
\end{lemma}

\begin{proof}
  By symmetry, it suffices to consider $\gth=\pi+\gf$ with
  $0\le\gf\le\pi/2$. Thus $z=-R_N\eigf$.
We have, by Taylor's formula,
\begin{equation}
  \label{la2a}
\frac23\xpar{-z-x}\qqc
= \frac23\xpar{-z}\qqc-x(-z)\qq+O\bigpar{R_N\qqw}
\end{equation}
and thus
\begin{equation}\label{la2b}
\Im \frac23\xpar{-z-x}\qqc
= \frac23R_N\qqc\sin\frac{3\gf}2-xR_N\qq \sin\frac{\gf}2+O\bigpar{R_N\qqw}.
\end{equation}
If $\gf\ge R_N\qqcw$, this shows that, for large $N$,
\begin{equation*}
\Im \frac23\xpar{-z-x}\qqc
\ge \frac4{3\sqrt2\,\pi} R_N\qqc \gf-\frac12 x R_N\qq\gf+O\bigpar{R_N\qqw}
\ge0.2,
\end{equation*}
and thus, by \eqref{Ai-z},
\begin{equation}\label{hag}
\begin{split}
    \lrabs{\Ai(z+x)}
&\asymp
|z+x|\qqqqw \lrabs{\sin\Bigpar{\tfrac23(-z-x)\qqc+\tfrac{\pi}4}}
\\&
\asymp
R_N\qqqqw \exp\bigpar{\Im\tfrac23(-z-x)\qqc}.
\end{split}
\end{equation}
In particular, taking $x=0$ we obtain \eqref{la2ii} for these $\gf$.
Moreover, comparing \eqref{hag} with the special case $x=0$, 
and using \eqref{la2b}, 
\begin{equation*}
  \lrabs{\frac{\Ai(z+x)}{\Ai(z)}}
\asymp
\exp\Bigpar{\Im\tfrac23(-z-x)\qqc-\Im\tfrac23(-z)\qqc}
\asymp
\exp\Bigpar{-xR_N\qq \sin\frac{\gf}2}
.
\end{equation*}
This yields \eqref{la2i} for $R_N\qqcw\le\gf\le\pi/2$, since 
$\sin(\gf/2)\ge\gf/\pi=|\gth-\pi|/\pi$.

For $0\le\gf\le R_N\qqcw$, we find from \eqref{la2b}
$\Im\frac23(-z-x)\qqc=O(1)$, 
so \eqref{Ai-z} yields $\Ai(z+x)=O(|z|\qqqqw)$.
Similarly, by our
choice of $R_N$, 
\begin{equation*}
  \begin{split}
\Re \Bigpar{\frac23\xpar{-z}\qqc}
&= \frac23R_N\qqc\cos\frac{3\gf}2
= \frac23R_N\qqc+O\bigpar{R_N\qqcw}
= \pi N+O\bigpar{R_N\qqcw}.	
  \end{split}
\end{equation*}
Hence, for large $N$ at least, 
$
\bigabs{\sin \bigpar{\frac23\xpar{-z}\qqc+\tfrac\pi4}}\asymp 1
$, 
and, by \eqref{Ai-z} again,
$\Ai(z)=\Theta(|z|\qqqqw)$; consequently
$\Ai(z+x)/\Ai(z)=O(1)$. These are the results we want in this case
because $R_n\qq\gf=O(1)$ and $R_n\qqc\gf=O(1)$.
\end{proof}

Recall also the entire functions $\Gi$ and $\Hi$ defined by
\AS{(10.4.44), (10.4.46)}:
\begin{align}\label{Hi}
  \Hi(z)&\=\pi\qw\intoo  \exp\Bigpar{-\frac13t^3+zt}\dd t
\\
\Gi(z)&\=\Bi(z)-\Hi(z).
\label{GBH}
\end{align}
The integral defining $\Hi$ evidently converges for all
complex $z$. There is a similar integral yielding $\Gi$ \AS{(10.4.42)}:
\begin{equation}
  \Gi(z)\=\pi\qw\intoo
  \sin\Bigpar{\frac13t^3+zt}\dd t,
\qquad z\in\bbR,
\end{equation}
but this integral converges only conditionally for real $z$, and
not at all for $z\notin\bbR$.

We have the asymptotics, as $x\to\infty$, \AS{(10.4.91)}
\begin{equation}
  \label{Hi-}
\Hi(-x)\sim\pi\qw x\qw.
\end{equation}
More generally, for $\Re z<0$,  
expanding $\exp(-t^3/3)$ in \eqref{Hi} yields, for any $L\ge0$,
\begin{equation}
  \label{Hizk}
\Hi(z)
=-\pi\qw\sum_{\ell=0}^{L-1}\frac{(3\ell)!}{3^\ell}z^{-3\ell-1}
+O\lrpar{|\Re z|^{-3L-1}}.
\end{equation}
In particular,
for any $\gd>0$ and $\arg z\in(\frac\pi2+\gd,\frac{3\pi}2-\gd)$,
\begin{equation}
  \label{Hiz}
\Hi(z)
=O\lrpar{|\Re z|^{-1}}
=O\lrpar{| z|^{-1}}.
\end{equation}
More precisely, for such $z$, 
\begin{equation}
  \label{Hiz4}
\Hi(z)
=-\pi\qw z\qw+O\lrpar{| z|^{-4}}.
\end{equation}

We further define, 
using $\intoo\Ai(x)\dd x=1/3$ \cite[10.4.82]{AS},
\begin{equation}
  \label{aAI}
\AI(z)\=\int_z^{+\infty} \Ai(t)\dd t
=\frac13-\int_0^z \Ai(t)\dd t;
\end{equation}
this is well-defined for all complex $z$ and yields an entire
function provided the first integral is taken along, for example, a
path that eventually follows the positive real axis to $+\infty$.
Note that $\AI'(z)=-\Ai(z)$ and that $\AI(0)=1/3$.
Along the real axis we have the limits,
by \cite[10.4.82--83]{AS},
\begin{align}
\lim_{x\to+\infty} \AI(x)&=0,
\label{AI+infty}
\\
\lim_{x\to-\infty} \AI(x)&=\int_{-\infty}^\infty \Ai(x)\dd x=1.
\label{AI-infty}
\end{align}

In terms of the functions $\Gi$ and $\Hi$, we
have, see \cite[10.4.47--48]{AS},
\begin{align}
  \AI(z)
&=\pi\bigpar{\Ai(z)\Gi'(z)-\Ai'(z)\Gi(z)}
\label{AIG}
\\
&=1+\pi\bigpar{\Ai'(z)\Hi(z)-\Ai(z)\Hi'(z)}.
\label{AIH}
\end{align}

We have, see \cite[Appendix A]{SJ201} 
and (for real $z$) \cite[10.4.82--83]{AS},
as $|z|\to\infty$ with $|\arg(z)|<\pi-\gd$,
\begin{equation}\label{AIooz}
    \AI(z)
\sim \frac1{2\sqrt \pi} z^{-3/4} e^{-\frac23z\qqc}, 
\end{equation}
and, along the negative real axis, and more generally for $-z$
with $|\arg(z)|<2\pi/3-\gd$, 
\begin{multline}\label{AI-z}
  \AI(-z)
= 1-\pi\qqw z^{-3/4}
\biggpar{\cos\Bigpar{\frac23z\qqc+\frac\pi4}\bigpar{1+O(|z|\qqcw)}
+O\bigpar{|z|\qqcw}}.
\end{multline}

We have the following estimates.

\begin{lemma}
  \label{LA3}
  \begin{thmenumerate}
\item	
For every fixed $\gd>0$, 
if $|\arg z|\le\pi-\gd$ and $|z|\ge1$, then
\begin{align*}
 \lrabs{\frac{\Ai'(z)}{\Ai(z)}} &\asymp |z|\qq,
\\
 \lrabs{\frac{\AI(z)}{\Ai(z)}} &\asymp |z|\qqw.
\end{align*}
\item	
For $z=R_N\eith$ with
$\gth\in[\frac\pi2,\frac{3\pi}2]$ and $N\ge1$,
\begin{align*}
 \lrabs{\frac{\Ai'(z)}{\Ai(z)}} 
&=O\lrpar{|z|\qq},
\\
 \lrabs{\frac{\AI(z)}{\Ai(z)}} 
&=O\lrpar{ |z|\qqw+|\Ai(z)|\qw}
=O\lrpar{ |z|\qqw+|z|\qqqq e^{-\frac23\lrabs{\Im(-z)\qqc}}}.
\end{align*}
  \end{thmenumerate}
\end{lemma}
\begin{proof}
  Part (i) follows from \eqref{Aioo}, \eqref{Aiioo} and \eqref{AIooz}.

For (ii), we use \eqref{la2ii}, \eqref{Ai'-z} and \eqref{AI-z}.
\end{proof}

Let $\glx_k$, $k\ge1$, denote the zeros of the Airy function; these are
all real and negative, so we have
$0>\glx_1>\glx_2>\dots$.
By \AS{(10.4.94), (10.4.105)},
\begin{equation}\label{gl}
  -\glx_k
=\Bigparfrac{3\pi(4k-1)}{8}\qqqb\bigpar{1+O(k\qww)}
\sim\parfrac{3\pi}2\qqqb k\qqqb 
\asymp k\qqqb.
\end{equation}
Thus, by \eqref{Ai'-}, see also \AS{(10.4.96)},
\eqref{Bioo-}, \eqref{Hi-} and \eqref{GBH},
\begin{align}
\label{Ai'glk}
  |\Ai'(\glx_k)|&\asymp |\glx_k|\qqqq \asymp k^{1/6},
\\
\label{Biglk}
  |\Bi(\glx_k)|&\asymp |\glx_k|\qqqqw \asymp k^{-1/6},
\\
\label{Higlk}
  |\Hi(\glx_k)|&\asymp |\glx_k|\qw \asymp k^{-2/3}.
\\
\label{Giglk}
  |\Gi(\glx_k)|&\asymp |\glx_k|\qqqqw \asymp k^{-1/6},
\end{align}

\section{Some Airy integrals}\label{AiryB}

Integrals of the Airy functions (and their derivatives) times powers
of $x$ are easily reduced using the relations
$\Ai''(x)=x\Ai(x)$ and $\Bi''(x)=x\Bi(x)$ and
integration by parts. We have, for example, using also the definition
\eqref{aAI}, 
\begin{align}
  \int \Ai(x)\dd x &= -\AI(x)
\label{IAi} \\
  \int x\Ai(x) \dd x&= \int \Ai''(x)\dd x=\Ai'(x)
\label{IxAi} \\
  \int x^2\Ai(x)\dd x &= \int x\Ai''(x)\dd x=x\Ai'(x)-\int \Ai'(x)\dd x
=x\Ai'(x)-\Ai(x)
\label{Ix2Ai} 
\end{align}
and in general the recursion
{\multlinegap0pt
\begin{multline}
\int x^n\Ai(x)\dd x=\int x^{n-1}\Ai''(x)\dd x
=x^{n-1}\Ai'(x)-(n-1)\int x^{n-2}\Ai'(x)\dd x
\\
=x^{n-1}\Ai'(x)-(n-1)x^{n-2}\Ai(x)
+(n-1)(n-2)\int x^{n-3}\Ai(x)\dd x.
\label{IxnAi}  
\end{multline}
}

Integrals of products of two Airy functions and powers of $x$
can be treated similarly, see \cite{Albright77}; we quote the
following (that are easily verified by differentiation):
\begin{align}
  \int \Ai(x)^2\dd x &=x\Ai(x)^2-(\Ai'(x))^2,
\label{IAi2} \\
  \int x\Ai(x)^2 \dd x&=\tfrac13\bigpar{x^2\Ai(x)^2-x(\Ai'(x))^2+\Ai'(x)\Ai(x)}
\label{IxAi2} \\
  \int x^2\Ai(x)^2\dd x &=
\tfrac15\bigpar{x^3\Ai(x)^2-x^2(\Ai'(x))^2+2x\Ai'(x)\Ai(x)-\Ai(x)^2}
\label{Ix2Ai2} 
\end{align}
and in general the recursion
{\multlinegap0pt
\begin{multline}
\int x^n\Ai(x)^2\dd x
=\tfrac{1}{2n+1}\biggl(
x^{n+1}\Ai(x)^2-x^{n}(\Ai'(x))^2+nx^{n-1}\Ai(x)\Ai'(x)
\\
-\frac{n(n-1)}2 x^{n-2}(\Ai(x))^2
+\frac{n(n-1)(n-2)}2\int x^{n-3}(\Ai(x))^2\dd x\biggr).
\label{IxnAi2}  
\end{multline}
}

By the same method, we can also treat products involving two different
translates of Airy functions; this gives for example the
following (again, these are easily verified by differentiation),
if $a\neq b$ and $c=(a+b)/2$:
\begin{align}
  \int \Ai(x+a)\Ai(x+b)\dd x &
=\frac{1}{a-b}\bigpar{\Ai'(x+a)\Ai(x+b)-\Ai(x+a)\Ai'(x+b)},
\label{IAi2ab} 
\end{align}
{\multlinegap0pt
\begin{multline}
  \int (x+c)\Ai(x+a)\Ai(x+b) \dd x
\\
\begin{aligned}
&=\frac1{(a-b)^2}\Bigl((a-b)(x+c)(\Ai'(x+a)\Ai(x+b)-\Ai(x+a)\Ai'(x+b))
\\
&\hskip6em
{}-2(x+c)\Ai(x+a)\Ai(x+b)+2\Ai'(x+a)\Ai'(x+b)\Bigr)
\\
&\qquad{}+\frac{2}{(a-b)^3}\bigpar{\Ai'(x+a)\Ai(x+b)-\Ai(x+a)\Ai'(x+b)},
\end{aligned}
\label{IxAi2ab}   
\end{multline}}
and the recursion
{\multlinegap0pt
\begin{multline}
\int (x+c)^n\Ai(x+a)\Ai(x+b)\dd x
=
\\
\begin{aligned}
\frac1{(a-b)^2}
\biggl(&
(a-b) (x+c)^n\bigpar{\Ai'(x+a)\Ai(x+b)-\Ai(x+a)\Ai'(x+b)}
\\&
-2n(x+c)^{n}\Ai(x+a)\Ai(x+b)
+2n(x+c)^{n-1}\Ai'(x+a)\Ai'(x+b)
\\&
-n(n-1)(x+c)^{n-2}\bigpar{\Ai'(x+a)\Ai(x+b)+\Ai(x+a)\Ai'(x+b)}
\\&
+n(n-1)(n-2)(x+c)^{n-3}\Ai(x+a)\Ai(x+b)
\\&
+
2n(2n-1)\int(x+c)^{n-1}\Ai(x+a)\Ai(x+b)\dd x
\\&
-
n(n-1)(n-2)(n-3)\int (x+c)^{n-4}\Ai(x+a)\Ai(x+b)\dd x
\biggr).
\end{aligned}
\label{IxnAi2ab}  
\end{multline}
}

In particular, if $a_k$ and $a_\ell$ are zeros of $\Ai$, with $k\neq\ell$,
the formulas above yield,
recalling the rapid decay \eqref{Aioo} and \eqref{Aiioo} at $\infty$:
\begin{align}
  \intoo \Ai(x+a_k)\dd x =
  \intakoo \Ai(x)\dd x &= \AI(a_k)
=-\pi\Ai'(a_k)\Gi(a_k),
\label{IAik} \\
  \intakoo x\Ai(x) \dd x&=-\Ai'(a_k),
\label{IxAik} \\
  \intakoo x^2\Ai(x)\dd x &
=-a_k\Ai'(a_k),
\label{Ix2Aik} 
\\
 \intoo x\Ai(x+a_k)\dd x &=-\Ai'(a_k)-a_k\AI(a_k),
\label{IxAikx} 
\\
  \intoo \Ai(x+a_k)^2\dd x =
  \intakoo \Ai(x)^2\dd x &=(\Ai'(a_k))^2,
\label{IAi2k} \\
  \intakoo x\Ai(x)^2 \dd x&=\tfrac13a_k(\Ai'(a_k))^2,
\label{IxAi2k} \\
  \intakoo x^2\Ai(x)^2\dd x &=
\tfrac15a_k^2(\Ai'(a_k))^2,
\label{Ix2Ai2k} 
\\
 \intoo x\Ai(x+a_k)^2\dd x &=-\tfrac23a_k(\Ai'(a_k))^2,
\label{IxAi2kx} 
\\
  \intoo \Ai(x+a_k)\Ai(x+a_\ell)\dd x &=0,
\label{IAi2abkl} \\
  \intoo x\Ai(x+a_k)\Ai(x+a_\ell) \dd x
&=-\frac2{(a_k-a_\ell)^2}\Ai'(a_k)\Ai'(a_\ell).
\label{IxAi2abkl}   
\end{align}

More generally, for arbitrary complex $a\neq b$, by \eqref{IAi2ab} and
\eqref{IxAi2ab}, again using \eqref{Aioo}, \eqref{Aiioo},
\begin{align}
  \intoo \Ai(x+a)\Ai(x+b)\dd x &
=\frac{1}{a-b}\bigpar{\Ai(a)\Ai'(b)-\Ai'(a)\Ai(b)},
\label{Iooab} 
\\\notag
 \intoo x\Ai(x+a)\Ai(x+b) \dd x
&=\frac{a+b}{(a-b)^2} \Ai(a)\Ai(b)
- \frac2{(a-b)^2}\Ai'(a)\Ai'(b)
\\
&\qquad{}+\frac{2}{(a-b)^3}\bigpar{\Ai(a)\Ai'(b)-\Ai'(a)\Ai(b)}.
\label{Iooxab}   
\end{align}
In particular,
\begin{align}
  \intoo \Ai(x)\Ai(x+a_k)\dd x &
=\frac{\Ai(0)\Ai'(a_k)}{-a_k}
\label{Ioo0k} 
\\
 \intoo x\Ai(x)\Ai(x+a_k) \dd x
&=- \frac2{a_k^2}\Ai'(0)\Ai'(a_k)
-\frac{2}{a_k^3}{\Ai(0)\Ai'(a_k)}.
\label{Ioox0k}   
\end{align}

\begin{remark}
  \label{RSturm}  
The Sturm--Liouville operator $Tf(x)=-f''(x)+xf(x)$ on $[0,\infty)$ with the
boundary condition $f(0)=0$ has 
the eigenfunctions $\Ai(x+a_k)$ with eigenvalues $-a_k=|a_k|$.
\eqref{IAi2abkl} thus expresses the orthogonality of the eigenfunctions
which also follows by general operator theory; in fact, the operator
$T$ is self-adjoint and 
these eigenfunctions form an orthogonal basis in $L^2(0,\infty)$.
The corresponding ON basis is
by \eqref{IAi2k} given by the functions
$\Ai(x+a_k)/\Ai'(a_k)$, $k\ge1$.

For example, \eqref{guy} is the expansion of $G(2\qqqw x)$ in this
ON basis, with coefficients $\pi\Hi(a_k)$, which yields
\eqref{Gparseval} by Parseval's formula.
Similarly, $\Ai(x)$ has the expansion (convergent in $L^2$) 
\begin{equation}
 \Ai(x)=\sumki \frac{\Ai(0)}{\abs{a_k}\Ai'(a_k)}\Ai(x+a_k), 
\end{equation}
where the coefficients are given by \eqref{Ioo0k}, and thus,
using \eqref{IAi2},
\begin{equation}
  \label{Aiparseval}
\frac{\Ai'(0)^2}{\Ai(0)^2}
= \intoo \parfrac{\Ai(x)}{\Ai(0)}^2\dd x = \sumki \frac{1}{a_k^2}.
\end{equation}
\end{remark}

We will also use 
the Laplace transform of the Airy function, which is easily found.
(Taking $z$ imaginary, we obtain the Fourier transform $e^{i\xi^3/3}$
of $\Ai$; this is sometimes taken as the definition of $\Ai$,
see \eg\ \cite[Definition 7.6.8]{Horm1}.)
\begin{lemma}
  \label{LV}
If\/ $\Re z>0$, then 
\begin{equation*}
  \intoooo e^{zt}\Ai(t)\dd t=e^{z^3/3}.
\end{equation*}
\end{lemma}
\begin{proof}
  By \eqref{Aioo} and \eqref{Ai-}, the integral
  converges absolutely for every $z$ with $\Re z>0$,
  and thus the integral is an analytic function of $z$ in the
  right halfplane, say $F(z)$.
We have, for $\Re z>0$, by $\Ai''(t)=t\Ai(t)$ and
  two integrations by parts,
  \begin{equation*}
	\begin{split}
F'(z)&=\intoooo t e^{zt}\Ai(t)\dd t	  
=\intoooo e^{zt}\Ai''(t)\dd t	  
\\&
=-\intoooo ze^{zt}\Ai'(t)\dd t	  
=\intoooo z^2e^{zt}\Ai(t)\dd t	  
=z^2F(z).
	\end{split}
  \end{equation*}
Hence, $F(z)=Ce^{z^3/3}$ for some $C$.

For $z>0$, another integration by parts yields
  \begin{equation}\label{byx}
	\begin{split}
F(z)&=\intoooo  e^{zt}\Ai(t)\dd t	  
=\intoooo ze^{zs}\int_s^\infty \Ai(t)\dd t\dd s	  
=\intoooo e^{u}\int_{u/z}^\infty \Ai(t)\dd t\dd u.	  
	\end{split}
  \end{equation}
Since $\AI(x)\=\int_x^\infty\Ai(t)\dd t\to0$ as
$x\to+\infty$ 
and $\AI(x)\to1$ as
$x\to-\infty$ by by \eqref{AI+infty}--\eqref{AI-infty},
dominated convergence shows that, letting $z\downto0$ in \eqref{byx},
\begin{equation*}
  C=\lim_{z\downto0} F(z)
=\int_{-\infty}^\infty  e^u\ett{u<0}\dd u 
=\int_{-\infty}^0  e^u\dd u =1.
\qedhere
\end{equation*}
\end{proof}

\begin{remark}\label{RAiinv}
  By Fourier inversion we find, for any $\gs>0$,
\begin{equation}\label{raiinv}
  \Ai(t)=\frac1{2\pi\ii}\intgs e^{-zt+z^3/3}\dd z.
\end{equation}
By analytic extension, this holds for any complex $t$.
\end{remark}

\section{An integral equation for $f_\tau(t)$}\label{AD}

We give here another approach, based on Daniels [personal communication, 1993],
to find the density function $f_\tau$
of the defect stopping time $\tau=\tau_x$, which was the basis of our
development in \refS{SpfT1}. Unfortunately, we have not succeeded to
make this approach rigorous, but we find it intriguing that it
nevertheless yields the right result, so we present it here as an
inspiration for further research.

Let as in \eqref{tau} be the first passage time of $W(t)$ to the
barrier $b(t)\=-x-t^2/2$, where $x>0$ is fixed.
Let $\ftt(t)$ be the (defect) density of $\tau$, and
$\phi(y;t)=e^{-y^2/2t}/\sqrt{2\pi t}$ the density of $W(t)$. 

The first entrance decomposition of $W(t)$ to the region $w<b(t)$
gives the integral equation (using the strong Markov property)
\begin{equation*}
  \phi(w;t)=\intot \ftt(u)\phi\bigpar{w-b(u);t-u}\dd u
\end{equation*}
for $w<b(t)$.
Letting $w\upto b(t)$ we get the equation
\begin{equation}\label{a1}
  \phi(b(t);t)=\intot \ftt(u)\phi\bigpar{b(t)-b(u);t-u}\dd u
\end{equation}
for $t>0$.
(Similar arguments using the last exit decomposition, which leads to
another functional equation involving also another unknown function,
are used by Daniels \cite{D74} and
Daniels and Skyrme \cite{DS85}.)
Since
\begin{equation*}
  b(t)-b(u)=(u^2-t^2)/2
= -(t-u)(t+u)/2,
\end{equation*}
we have by \eqref{a1}
\begin{equation}\label{a2}
\frac{e^{-(x+t^2/2)^2/2t}}{\sqrt{2\pi t}}
=\intot \ftt(u)\frac{e^{-(t-u)(t+u)^2/8}}{\sqrt{2\pi(t-u)}}\dd u.
\end{equation}
The exponents can be written as
$-(t-u)(t+u)^2/8=-t^3/6+u^3/6+(t-u)^3/24$ and 
$-(x+t^2/2)^2/2t = -x^2/2t-xt/2+t^3/24-t^3/6$,
so the integral equation \eqref{a2} can be transformed into
\begin{equation}\label{a3}
\frac{e^{-x^2/2t-xt/2+t^3/24}}{\sqrt{2\pi t}}
=\intot \ftt(u)e^{u^3/6}\frac{e^{(t-u)^3/24}}{\sqrt{2\pi(t-u)}}\dd u.
\end{equation}
This is a convolution equation of the form
\begin{equation*}
  h(t)=\intot g(u)k(t-u)\dd u
\end{equation*}
with 
\begin{align*}
  g(t)&\= \fttt e^{t^3/6}, \\
h(t)&\= \frac{e^{-x^2/2t-xt/2+t^3/24}}{\sqrt{2\pi t}}
=\frac{e^{-b(t)^2/2t+t^3/6}}{\sqrt{2\pi t}}, \\
k(t)&\= \frac{e^{t^3/24}}{\sqrt{2\pi t}}.
\end{align*}
If the Laplace transforms
$
  \tg(s)\=\intoo e^{-st} g(t)\dd t
$ 
\etc\ were finite for $\Re s$ large enough, we could get the solution
from $\tg(s)=\thx(s)/\tk(s)$.
However, the factors $e^{t^3/24}$ in $h(t)$ and $k(t)$ grow too fast,
so $\thx(s)$ and $\tk(s)$ are not finite for any $s>0$ and this method
does not work.
Nevertheless, if we instead define $\hthx(s)$ and $\htk(s)$ by 
$\hthx(s)\=\intgs e^{-st} h(t) \dd t$ and $\htk(s)\=\intgs e^{-st} k(t) \dd t$,
integrating along vertical lines in the complex plane with real part $\gs>0$,
then the  formula $\tg(s)=\hthx(s)/\htk(s)$ yields the correct
formula for $\tg(s)$ and thus for $g(t)$. (Note that $\hthx$ and $\htk$ 
can be seen as Fourier transform of $h$ and $k$ restricted to vertical
lines. The value of $\gs>0$ is arbitrary and does not affect $\hthx$
and $\htk$.)
Let us show this remarkable fact by calculating $\hthx(s)$ and
$\htk(s)$.

Consider first $h(t)$ and express the Gaussian factor by Fourier
inversion:
\begin{equation*}
\frac{e^{-b(t)^2/2t}}{\sqrt{2\pi t}}
=\intgs e^{b(t)u+tu^2/2}\frac{\dd u}{2\pi\ii}, 
\qquad \Re t>0.
\end{equation*}
The exponent $b(t)u+tu^2/2$ can then be written as
$-ux+u^3/6+(t-u)^3/6-t^3/6$, so that
\begin{equation*}
h(t)
=\intgs e^{-ux+u^3/6+(t-u)^3/6}\frac{\dd u}{2\pi\ii}, 
\end{equation*}
and, choosing $\gs_1>\gs>0$ and letting $\gs_2\=\gs_1-\gs$,
\begin{equation*}
  \begin{split}
\hthx(s)
&=\intx{\gs_1}\intgs e^{-st-ux+u^3/6+(t-u)^3/6}\frac{\dd u\dd t}{2\pi\ii} 
\\
&=\intx{\gs}\intx{\gs_2} e^{-s(u+v)-ux+u^3/6+v^3/6}\frac{\dd v\dd u}{2\pi\ii} 
\\
&=\intx{\gs}\intx{\gs_2} e^{-sv+v^3/6} e^{-(s+x)u+u^3/6}
\frac{\dd v\dd u}{2\pi\ii} 
\\
&=2\pi\ii\,\Ai\bigpar{c(s+x)}\Ai(cs)c^2,
  \end{split}
\end{equation*}
with $c\=2\qqq$, using \eqref{raiinv}.

Since $k(t)$ is obtained by putting $x=0$ in $h(t)$, we get directly
\begin{equation*}
  \htk(s)=2\pi\ii\,\Ai(cs)^2c^2
\end{equation*}
and hence
\begin{equation*}
\hthx(s)/ \htk(s)=\Ai\bigpar{2\qqq (s+x)}/\Ai(2\qqq s).
\end{equation*}
This is indeed the Laplace transform of $g(t)=\fttt e^{t^3/6}$
given in \eqref{lt},
which by inversion yields the formulas \eqref{ft2} and \eqref{ft} for $\fttt$.

It seems likely that it should be possible to verify 
the crucial formula $\tg(s)\htk(s)=\hthx(s)$ by suitable manipulations
of integrals, which would give another proof of the formulas
\eqref{ft}--\eqref{lt} for $f_\tau(t)$ and $g(t)$.
For example, if we define, for $\Re t>0$, 
$$F(t)\=h(t)-\int_0^{\Re t}g(u)k(t-u)\dd u$$
(note that $F$ is not analytic),
then $F(t)=0$ for real $t>0$, and it is easily verified that the
equation $\tg(s)\htk(s)=\hthx(s)$ is equivalent to 
$\intgs e^{-st}F(t)\dd t\to0$ as $\gs\to\infty$.
However, we do not know how to verify this directly. 
We therefore leave finding a direct proof of
$\tg(s)\htk(s)=\hthx(s)$  as an open problem. 

\section{An alternative derivation of \eqref{lt}} 
\label{Aalt}

A proof of the formula \eqref{lt} for the Laplace transform of the
density of 
the passage time $\tau$ is given 
by \citet{Groeneboom} 
(in a more general form, allowing
an arbitrary starting point and not just $t=0$, or, equivalently, 
linear term $b t$ in \eqref{wg} or \eqref{tau}; for simplicity we do
not consider this extension). 
His proof uses partly quite technical methods.
For the service of the reader we here present an alternative proof based on the
same ideas but from a different point of view;
we believe that this yields
a more straightforward proof for our purposes.
As discussed in \refR{RFT}, this implies \eqref{ft}  and \eqref{lj1},
so it gives us a self-contained proof of the central \refL{LJ1}.

As in \cite{Salminen,ML98} we consider the process with drift $t^2/2$
defined by $X(t)\=x+t^2/2+W(t)$, so that $\tau$ defined by \eqref{tau}
(with $\gb=1/2$ as in \eqref{lj1}--\eqref{lt})
is the first hitting time of $X(t)=0$.
If $P_x$ is the probability measure (on the space $C[0,\infty)$)
corresponding to $X(\cdot)$, and $Q_x$ that corresponding to $x+W(\cdot)$,
then the Cameron--Martin formula tells us that, considering the
restriction to a finite time interval $[0,t]$, the Radon--Nikodym derivative
is, using $\dd X(s)=s\dd s+\dd W(s)$,
\begin{equation*}
  \begin{split}
\frac{\dd P_x(X)}{\dd Q_x(X)}	
&=\exp\lrpar{-\frac12\intot  
\lrpar{\frac{(\dd X(s)-s\dd s)^2}{\dd s}
-\frac{\dd X(s)^2}{\dd s}
}}
\\&
=\exp\lrpar{\intot s\dd X(s)-\frac12\intot s^2\dd s} 
\\&
=\exp\lrpar{-\frac{t^3}6+tX(t)-\intot X(s)\dd s} ,
  \end{split}
\end{equation*}
using integration by parts. Hence, letting $\Ex$ denote expectation
with respect to the Wiener measure $Q_x$,
\begin{equation}\label{d1}
  \begin{split}
P_x(\tau>t) = \Ex\Bigpar{\frac{\dd P_x}{\dd Q_x}; \tau>t}	
=e^{-t^3/6} \Ex\lrpar{e^{tX(t)-\intot X(s)\dd s};\tau>t}
  \end{split}
\end{equation}
(\cf{} \cite[Theorem 2.1]{Salminen} and \cite[Lemma 2.1]{Groeneboom}).
We introduce the ``Green function''
\begin{equation*}
  F(x,y,t)\=\Ex\lrpar{e^{-\intot X(s)\dd s}\gd\bigpar{X(t)-y};\tau>t},
\end{equation*}
where $\gd$ is the Dirac delta function (formally, $F(x,\cdot,t)$ is defined as
the density of the corresponding occupation measure); \eqref{d1} then
says that
\begin{equation}\label{d2}
  \begin{split}
P_x(\tau>t) 
=e^{-t^3/6} \intoo F(x,y,t)e^{ty}\dd y.
  \end{split}
\end{equation}
The Feynman--Kac formula tells us that for fixed $y>0$,
$F(x,y,t)$ is the fundamental
solution of the equation
\begin{equation}\label{FK}
  \frac{\partial F}{\partial t}
=\frac12 \frac{\partial^2 F}{\partial x^2}-xF,
\qquad x,t>0,
\end{equation}
with the boundary conditions $F(0,y,t)=F(\infty,y,t)=0$. 
By time-reversal (or by symmetry of the resolvent $R(x,y,z)$ 
in \eqref{r} below),
$F(x,y,t)=F(y,x,t)$; hence also 
\begin{equation}\label{FKy}
  \frac{\partial F}{\partial t}
=\frac12 \frac{\partial^2 F}{\partial y^2}-yF,
\qquad x,y,t>0,
\end{equation}

Differentiating \eqref{d2} under the integral sign 
and using the Feynman--Kac equation \eqref{FKy}, we obtain
(cf.\ \cite[Lemma 2.2]{Groeneboom})
\begin{equation}\label{d5}
  \begin{split}
e^{t^3/6}f_\tau(t)
&=-e^{t^3/6}\frac{\dd}{\dd t}\lrpar{
e^{-t^3/6} \intoo F(x,y,t)e^{ty}\dd y}
\\&
=\intoo\lrpar{\frac{t^2}2F(x,y,t)-\frac{\partial F(x,y,t)}{\partial	t}
-yF(x,y,t)}e^{ty}\dd y
\\&
=\intoo\lrpar{\frac{t^2}2F(x,y,t)
-\frac12\frac{\partial^2 F(x,y,t)}{\partial	y^2}
}e^{ty}\dd y
\\&
=\frac{1}2\lrsqpar{te^{ty}F(x,y,t)-e^{ty}\frac{\partial F(x,y,t)}{\partial	y}
}_0^\infty
\\&
=\frac{1}2\frac{\partial F}{\partial	y}(x,0,t),
  \end{split}
\end{equation}
since  $F(x,0,t)=0$ and $F$ and its derivatives decrease
rapidly as $y\to\infty$.

The Laplace transform of $F(x,y,\cdot)$ is the resolvent, defined at
least for $\Re z>0$, 
\begin{equation}\label{rdef}
  R(x,y,z)\=\intoo e^{-zt}F(x,y,t)\dd t.
\end{equation}
The Feynman--Kac equation is a second order differential equation,
and the theory of such equations 
(see also \cite[Appendix C; in particular (379)]{SJ201} for this
particular case)
tells us that 
\begin{equation}\label{r}
  R(x,y,z)=
  \begin{cases}
	\frac2w\gfo(x;z)\gfoo(y;z), & 0<x\le y, \\
	\frac2w\gfo(y;z)\gfoo(x;z), & 0<y\le x,
  \end{cases}
\end{equation}
where $\gfo(x)=\gfo(x;z)$ and $\gfoo(x)=\gfoo(x;z)$ are solutions of
the differential equation 
\begin{equation}
  \label{d3a}
\tfrac12\gf''(x)-x\gf(x)=z\gf(x)
\end{equation}
with boundary conditions $\gfo(0)=0$, $\gfoo(\infty)=0$, and $w$ is
the Wronskian
$w\=\gfoo(x)\gfo'(x)-\gfo(x)\gfoo'(x)$ (which is constant in $x$).
Note that since $\gfo(0)=0$, 
\begin{equation}
  \label{w}
w=\gfoo(0)\gfo'(0).
\end{equation}

The differential equation \eqref{d3a} has two linearly independent solutions
$A(x+z)$ and $B(x+z)$ with $A(x)\=\Ai(cx)$ and $B(x)\=\Bi(cx)$ with
$c\=2\qqq$, and in terms of these we have (up to arbitrary constant factors)
\begin{align}
  \gfo(x;z)&=A(z)B(x+z)-B(z)A(x+z),\\
\gfoo(x;z)&={A(x+z)}. \label{gfoo}
\end{align}

We integrate \eqref{d5} (multiplied by $e^{-zt}$) and differentiate
\eqref{rdef} and obtain, using also \eqref{r}, \eqref{w} and \eqref{gfoo}
\begin{equation*}
  \begin{split}
\intoo e^{-zt}e^{t^3/6}f_\tau(t)\dd t
&=\frac{1}2\intoo e^{-zt}\frac{\partial F}{\partial	y}(x,0,t)\dd t
=\frac{1}2\frac{\partial R}{\partial	y}(x,0,z)
\\&
=\frac{1}w\gfo'(0;z)\gfoo(x;z)
=\frac{\gfoo(x;z)}{\gfoo(0;z)}
\\&
=\frac{A(x+z)}{A(z)},
  \end{split}
\end{equation*}
which is \eqref{lt}.

\newcommand\AAP{\emph{Adv. Appl. Probab.} }
\newcommand\JAP{\emph{J. Appl. Probab.} }
\newcommand\JAMS{\emph{J. \AMS} }
\newcommand\MAMS{\emph{Memoirs \AMS} }
\newcommand\PAMS{\emph{Proc. \AMS} }
\newcommand\TAMS{\emph{Trans. \AMS} }
\newcommand\AnnMS{\emph{Ann. Math. Statist.} }
\newcommand\AnnPr{\emph{Ann. Probab.} }
\newcommand\CPC{\emph{Combin. Probab. Comput.} }
\newcommand\JMAA{\emph{J. Math. Anal. Appl.} }
\newcommand\RSA{\emph{Random Struct. Alg.} }
\newcommand\ZW{\emph{Z. Wahrsch. Verw. Gebiete} }
\newcommand\DMTCS{\jour{Discr. Math. Theor. Comput. Sci.} }

\newcommand\AMS{Amer. Math. Soc.}
\newcommand\Springer{Springer-Verlag}
\newcommand\Wiley{Wiley}

\newcommand\vol{\textbf}
\newcommand\jour{\emph}
\newcommand\book{\emph}
\newcommand\inbook{\emph}
\def\no#1#2,{\unskip#2, no. #1,} 
\newcommand\toappear{\unskip, to appear}

\newcommand\webcite[1]{
\texttt{\def~{{\tiny$\sim$}}#1}\hfill\hfill}
\newcommand\webcitesvante{\webcite{http://www.math.uu.se/~svante/papers/}}
\newcommand\arxiv[1]{\webcite{arXiv:#1.}}

\def\nobibitem#1\par{}

\end{document}